\documentclass{amsart}

\usepackage{mathrsfs}
\usepackage{amsfonts}
\usepackage{}

\usepackage{amsmath}
\usepackage{amscd}
\usepackage{amssymb}

\newcommand{\cal}{\mathcal}

\newcommand{\bC}{{\Bbb C}}

\newcommand{\bL}{{\Bbb L}}

\newcommand{\bP}{{\Bbb P}}

\newcommand{\bZ}{{\Bbb Z}}

\newcommand{\cB}{{\cal B}}

\newcommand{\cD}{{\cal D}}

\newcommand{\cM}{{\cal M}}
\newcommand{\cO}{{\cal O}}

\newcommand{\cT}{{\cal T}}

\newcommand{\cU}{{\cal U}}

\newcommand{\cX}{{\cal X}}
\newcommand{\cY}{{\cal Y}}
\newcommand{\Mbar}{\overline{\cM}}
\newcommand{\mubar}{\overline{\mu}}
\newcommand{\nubar}{\overline{\nu}}
\newcommand{\xibar}{\overline{\xi}}
\newcommand{\etabar}{\overline{\eta}}

\DeclareMathOperator{\Gcd}{gcd}
\DeclareMathOperator{\Aut}{Aut}

\DeclareMathOperator{\rk}{rk}

\newcommand{\vir}{ {\mathrm{vir}} }

\newcommand{\tpi  }{\tilde{\pi}  }

\newcommand{\xn}[1]{{#1}^{\bullet s}_{\chi,\mubar,\gamma}}

\newcommand{\lam}{\lambda}

\newcommand{\bu}{\bullet}

\newcommand{\amm}{|\Aut(\nubar)||\Aut(\mubar)|}
\newcommand{\am}{|\Aut(\mubar)|}

\newtheorem{theorem}{Theorem}[section]
\newtheorem{Theorem}{Theorem}
\newtheorem{proposition}[theorem]{Proposition}
\newtheorem{lemma}[theorem]{Lemma}

\theoremstyle{remark}

\theoremstyle{definition}

\newtheorem{definition}{Definition}[section]

\begin{document}
\title[A Formula of the One-leg Orbifold Gromov-Witten Vertex]{A Formula of the One-leg Orbifold Gromov-Witten Vertex and Gromov-Witten Invariants of the Local $\cB\bZ_m$ Gerbe}
\author{Zhengyu Zong}
\address{Department of Mathematics, Columbia University, New York, NY 10027, USA}
\email{zz2197@math.columbia.edu}
\begin{abstract}
We give a formula of the framed one-leg orbifold Gromov-Witten vertex where the leg is gerby with isotropy group $\bZ_m$. Then we use this formula to compute the Gromov-Witten invariants of the local $\cB\bZ_m$ gerbe. We will also compute some examples of the degree 1 and degree 2 $\bZ_2$-Hodge integrals.
\end{abstract}
\maketitle
\section{Introduction}
For smooth toric Calabi-Yau 3-folds, the Gromov-Witten theory is obtained by gluing the Gromov-Witten topological vertex \cite{AKMV}, a generating function of cubic Hodge integrals. So the topological vertex, which is computed in \cite{Li-Liu-Liu-Zhou}, can be viewed as the building block for the GW theory of smooth toric Calabi-Yau 3-folds. In the orbifold case, a vertex formalism for the orbifold GW theory of toric Calabi-Yau 3-orbifolds is established in \cite{Ros}. For toric Calabi-Yau 3-orbifolds, the orbifold GW theory is obtained by gluing the GW orbifold vertex, a generating function of cubic abelian Hurwitz-Hodge integrals. So the orbifold GW vertex can be viewed as the building block of the orbifold GW theory of toric Calabi-Yau 3-orbifolds.

In this paper, we prove a formula of the framed one-leg orbifold Gromov-Witten vertex where the leg is gerby with isotropy group $\bZ_m$. This formula can be viewed as a counterpart of the case where the leg is effective \cite{Zong}. Unlike the effective case, the initial value of our vertex can not be computed by Mumford's relation. So we will use virtual localization and vanishing properties of certain relative GW invariants of the nontrivial $\bZ_m$-gerbes over $\bP^1$ to obtain a system of linear equations, from which we can solve the initial value. After that we will use localization on certain moduli spaces of relative stable morphisms to the trivial $\bZ_m$-gerbes over $\bP^1$ to obtain the framing dependence of the vertex. We also use our formula to compute the GW invariants of the local $\cB\bZ_m$ gerbe.

In section 5, we will calculate the degree 1 and degree 2 $\bZ_2$-vertices more explicitly. Then we will use these results to calculate the prediction of some $\bZ_2$-Hodge integrals in \cite{Ros}. These results can be viewed as an evidence for the conjecture of the orbifold GW/DT correspondence in \cite{Ros}.

After the first version of this paper, our formula was used to prove the orbifold GW/DT correspondence in \cite{Ros2}. The strategy there is to rewrite the DT vertex in terms of the loop schur functions developed in \cite{Ros3} and \cite{Lam-Pyl} to obtain some useful combinatorial properties of the DT vertex. Then since our formula for the GW vertex involves a certain kind of orbifold rubber integrals which have strongly combinatorial properties (see section 3), one can prove the GW/DT correspondence by proving certain purely combinatorial identities.

\subsection{Formula for the vertex}
\subsubsection{Framing dependence of the vertex}
Let $\Mbar_{g, \gamma}(\cB\bZ_m)$ be the moduli space of stable maps to $\cB\bZ_m$ where $\gamma=(\gamma_{1}, \cdots, \gamma_{n})$ is a vector of elements in $\bZ_m$. Let $U$ be the irreducible representation of $\bZ_m$ given by
$$\phi^U:\bZ_m\to \bC^*,\phi^U(1)=e^{\frac{2\pi\sqrt{-1}}{m}}$$
Then there is a corresponding Hodge bundle
$$E^U\to\Mbar_{g, \gamma}(\cB\bZ_m)$$
and the corresponding Hodge classes on $\Mbar_{g, \gamma}(\cB\bZ_m)$ are defined by Chern classes of $E^U$,
$$\lambda^U_i=c_i(E^U)$$
Similarly, for any irreducible representation $R$ of $\bZ_m$, we have a corresponding Hodge bundle $E^R$ and Hodge classes $\lambda^R_i$. Let $\Mbar_{g,n}$ be the moduli space of stable curves of genus $g$ with $n$ marked points and let $\psi_i$ be the $i^{th}$ descendent class on $\Mbar_{g,n}$, $1\leq i\leq n$. Let
$$\epsilon :\Mbar_{g, \gamma}(\cB\bZ_m)\to \Mbar_{g,n}$$
be the canonical morphism. Then the descendent classes $\bar{\psi}_i$ on $\Mbar_{g,n}$ are defined by
$$\bar{\psi}_i=\epsilon^*(\psi_i)$$
Let
$$
\Lambda^{\vee,R}_g(u)=u^{\rk E^R}-\lambda^R_1 u^{\rk E^R-1} +\cdots+(-1)^{\rk E^R}\lambda^R_{\rk E^R}
$$
where $\rk E^R$ is the rank of $E^R$ determined by the orbifold Riemann-Roch formula.

Let $d$ be a positive integer and let
$$
\mubar=\{(\mu_1,k_1),\cdots , (\mu_{l(\mu)}, k_{l(\mu)})\}
$$
be a $\bZ_m$-weighted partition of d. Here, $\mu$ is a partition of $d$ with parts $\mu_i$ and $k_i\in \bZ_m$. Let
$$\{1,\cdots, l(\mubar)\}=A'(\mubar)\sqcup A''(\mubar)$$
such that $k_i=0$ if and only if $i\in A'(\mubar)$. Let $l'(\mubar)=|A'(\mubar)|$ and $l''(\mubar)=|A''(\mubar)|$. For any $\mubar$, we use the notation $-\mubar$ to denote $\{(\mu_1,-k_1),\cdots , (\mu_{l(\mu)}, -k_{l(\mu)})\}$.

Now we require $\gamma=(\gamma_{1}, \cdots, \gamma_{n})$ be a vector of \emph{nontrivial} elements in $\bZ_m$. Then for $\tau\in \frac{1}{m}\bZ$, we define $G_{g,\mubar,\gamma}(\tau)_{m}$ as
\begin{eqnarray*}
&&\frac{(\sqrt{-1})^{|\mubar|+l(\mubar)-2\sum_{i\in A''(\mubar)}\frac{m-k_i}{m}}
\left(\tau(\tau+1)\right)^{\sum_{i=1}^{l(\mubar)}\delta_{0,k_i}}}{|\Aut(\mubar)|}
\prod_{i=1}^{l(\mubar)}\frac{\prod_{j=0}^{\mu_i-1}(\mu_i\tau+\frac{k_i}{m}+j)}{\mu_i!}
\left(\frac{\mu_i\tau+\frac{k_i}{m}}{\mu_i}\right)^{-\delta_{0,k_i}}\\
&&\cdot\int_{\Mbar_{g, \gamma+k}(\cB\bZ_m)}\frac{(-\tau(\tau+1))^{-\delta}\Lambda_{g}^{\vee,U}(\tau)\Lambda_{g}^{\vee,U^\vee}(-\tau-1)
\Lambda_{g}^{\vee,1}(1)}{\prod_{i=1}^{l(\mubar)}(1-\mu_i\bar{\psi}_i)}
\end{eqnarray*}
where $\gamma+k$ denotes the vector $(\gamma_{1}, \cdots, \gamma_{n},k_1,\cdots,k_{l(\mubar)})$, $\bar{\psi}_i$ corresponds to $k_i$, $U^\vee$ and 1 denote the dual of $U$ and the trivial representation respectively,
$$\delta_{0,x}=\left\{\begin{array}{ll}1, &x=0,\\
0, &x\neq 0,\end{array} \right.$$
and
$$\delta=\left\{\begin{array}{ll}1, &\textrm{if all monodromies around loops on the domain curve are trivial}\\
0, &\textrm{otherwise}.\end{array} \right.$$

Introduce formal variables $p=(p_{(i,j)})_{i\in\bZ_+,j\in\{0,\cdots,m-1\}},x=(x_1,\cdots,x_{m-1})$ and define
$$p_{\mubar}=p_{(\mu_1,k_1)}\cdots p_{(\mu_{l(\mubar)},k_{l(\mubar)})},x_{\gamma}=x_{\gamma_1}\cdots x_{\gamma_n}$$
for $\mubar$ and $\gamma$. We use the more intuitive symbol $\gamma!$ to denote $|\Aut(\gamma)|$. Then we define the generating functions
\begin{eqnarray*}
G_{\mubar,\gamma}(\lambda;\tau)_{m}&=&\sum_{g=0}^{\infty}\lambda^{2g-2+l(\mubar)+l(\gamma)}G_{g,\mubar,\gamma}(\tau)_{m}\\
G(\lambda;\tau;p;x)_{m}&=&\sum_{\mubar\neq \emptyset,\gamma}G_{\mubar,\gamma}(\lambda;\tau)_{m}p_{\mubar} \frac{x_\gamma}{\gamma!}=\sum_{\mubar\neq \emptyset}G_{\mubar}(\lambda;\tau;x)_{m}p_{\mubar}=\sum_{\gamma}G_{\gamma}(\lambda;\tau;p)_{m}\frac{x_\gamma}{\gamma!}\\
G^\bullet(\lambda;\tau;p;x)_{m}&=&\exp(G(\lambda;\tau;p;x)_{m})=
\sum_{\mubar,\gamma}G^\bullet_{\mubar,\gamma}(\lambda;\tau)_{m}p_{\mubar} \frac{x_\gamma}{\gamma!}=
1+\sum_{\mubar\neq\emptyset}G^\bullet_{\mubar}(\lambda;\tau;x)_{m}p_{\mubar}\\
G^\bullet_{\mubar,\gamma}(\lambda;\tau)_{m}&=&\sum_{\chi\in 2\bZ,\chi\leq 2l(\mubar)}\lambda^{-\chi+l(\mubar)+l(\gamma)}G^\bullet_{\chi,\mubar,\gamma}(\tau)_{m}
\end{eqnarray*}
When taking the sum over $\gamma$, we set the following convention for $\gamma$: if $i<j$, then $\gamma_j\geq \gamma_i$ where $\gamma_1,\cdots, \gamma_n$ are viewed as elements in $\{1,\cdots ,m-1\}$.

In section 3, we will define an orbifold rubber integral $H^\bullet_{\chi,\gamma}(\mubar,\nubar)_m$ and its generating function
$$\Phi^\bullet_{\mubar,\nubar}(\lambda;x)_m=\sum_{\chi,\gamma}\frac{\lambda^{-\chi+l(\mubar)+l(\nubar)+l(\gamma)}}
{(-\chi+l(\mubar)+l(\nubar))!}H^\bullet_{\chi,\gamma}(\mubar,\nubar)_m\frac{x_\gamma}{\gamma!}$$
Define $G_d(\tau)=(G^\bullet_{\nubar}(\lambda;\tau;x)_m)_{|\nubar|=d}$ and $G_d(0)=(G^\bullet_{\mubar}(\lambda;0;x)_m)_{|\mubar|=d}$ to be two column vectors indexed by $\nubar$ and $\mubar$ respectively. Let $\Phi_d(\tau)=(\Phi_d^{\mubar,\nubar}(\tau))_{|\mubar|=d,|\nubar|=d}$ be a matrix indexed by $\nubar$ and $\mubar$, where
$$\Phi_d^{\mubar,\nubar}(\tau)=Z_{\nubar}
\Phi^\bullet_{-\nubar,\mubar}(-\sqrt{-1}\tau\lambda;\sqrt{-1}^{1-\frac{2}{m}}x_1,\cdots,\sqrt{-1}^{1-\frac{2i}{m}}x_i,
\cdots,\sqrt{-1}^{1-\frac{2(m-1)}{m}}x_{m-1})_m.$$
and $Z_{\nubar}=|\Aut(\nubar)|m^{l(\nubar)}\prod_{i=1}^{l(\nubar)}\nu_i$. Then we have the following theorem
\begin{Theorem}
$\Phi_d(\tau)$ is invertible and
\begin{eqnarray*}
G_d(\tau)=\Phi_d(\tau)^{-1}G_d(0),
\end{eqnarray*}
or equivalently,
\begin{eqnarray*}
G^\bullet_{\mubar}(\lambda;\tau;x)_m=\sum_{|\nubar|=|\mubar|}G^\bullet_{\nubar}(\lambda;0;x)_m Z_{\nubar}
\Phi^\bullet_{-\nubar,\mubar}(\sqrt{-1}\tau\lambda;\tilde{x})_m.
\end{eqnarray*}
where $\tilde{x}=(\sqrt{-1}^{1-\frac{2}{m}}x_1,\cdots,\sqrt{-1}^{1-\frac{2i}{m}}x_i,
\cdots,\sqrt{-1}^{1-\frac{2(m-1)}{m}}x_{m-1})$.
\end{Theorem}

We will calculate $H^\bullet_{\chi,\gamma}(\mubar,\nubar)_m$ and $\Phi^\bullet_{\mubar,\nubar}(\lambda;x)_m$ in section 3 in two different ways. One way is to relate $H^\bullet_{\chi,\gamma}(\mubar,\nubar)_m$ to usual double Hurwitz numbers and the other way is to give $H^\bullet_{\chi,\gamma}(\mubar,\nubar)_m$ an intrinsic combinatorial expression using representations of the wreath product. Concretely, we will prove the following theorem:\\
\\
\textbf{Theorem 3.5} (Burnside type formula for $H^\bullet_{\chi,\gamma}(\mubar,\nubar)_m$):
$$H^\bullet_{\chi,\gamma}(\mubar,\nubar)_m=\sum_{|\xibar|=d}
\frac{X_{\xibar}(\mubar)}{Z_{\mubar}}\frac{X_{\xibar}(\nubar)}{Z_{\nubar}}
F_{\xibar}(\tau_0)^r\prod_{i=1}^{l(\gamma)}F_{\xibar}(\rho_{\gamma_i}),$$
where $F_{\xibar}(\tau_0)=\sum_{i=0}^{m-1}\frac{m\kappa_{\xibar(i)}}{2}$ and
$F_{\xibar}(\rho_{\gamma_i})=\sum_{j=0}^{m-1}|\xibar(j)|
e^{\frac{-2\gamma_ij\pi\sqrt{-1}}{m}}$.

Theorem 3.5 implies that
$$\Phi^\bullet_{\mubar,\nubar}(\lambda;x)_m=\sum_{|\xibar|=d}
\frac{X_{\xibar}(\mubar)}{Z_{\mubar}}\frac{X_{\xibar}(\nubar)}{Z_{\nubar}}e^{F_{\xibar}(\tau_0)\lambda}
\prod_{i=1}^{m-1}e^{F_{\xibar}(\rho_{i})x_i\lambda}$$
So Theorem 1 expresses our framed orbifold Gromov-Witten vertex $G^\bullet_{\mubar}(\lambda;\tau;x)_m$ in terms of its initial value $G^\bullet_{\mubar}(\lambda;0;x)_m$ and the rubber integral $H^\bullet_{\chi,\gamma}(\mubar,\nubar)_m$.

\subsubsection{Calculation of the initial value}
For any integer $d\geq 1, s\in \{1,\cdots,m-1\}$ and $l\in\bZ_m$, let
$$w^s_{d}(l)=-l-ds\in \bZ_m.$$
Similarly, for any $\bZ_m$-weighted partition $\nubar=\{(\nu_1,l_1),\cdots,(\nu_{l(\nu)},l_{l(\nu)})\}$, let
$$w^s(\nubar)=\{(\nu_1,w^s_{\nu_1}(l_1)),\cdots,(\nu_{l(\nu)},w^s_{\nu_{l(\nu)}}(l_{l(\nu)}))\}.$$
Let
\begin{eqnarray*}
B_d&=&\{\etabar||\etabar|\leq d,l(\etabar)=l''(\etabar)\}\\
C_d&=&\{(\mubar,s)|\mu=\eta,s=s(\etabar),k_1=0,-w^s(\mubar\setminus \{(\mu_1,k_1)\})=\etabar\setminus \{(\eta_1,h_1)\},|\etabar|\leq d,l(\etabar)=l''(\etabar)\}.
\end{eqnarray*}
where $s=s(\etabar)$ is a function that we will explain in section 4.4. Let $$\tilde{x}=(\sqrt{-1}^{1-\frac{2}{m}}x_1,\cdots,\sqrt{-1}^{1-\frac{2i}{m}}x_i,
\cdots,\sqrt{-1}^{1-\frac{2(m-1)}{m}}x_{m-1}),$$
then we define
$$\beta_d=(-\sum_{|\xibar|=|\mubar|,l(\xibar)=l'(\xibar)}G^\bullet_{\xibar}(\lambda;0;x)_mZ_{w^s(\xibar)}
\tilde{\Phi}^\bullet_{-w^s(\xibar),\mubar}(-\frac{\sqrt{-1}s}{m}\lambda;\tilde{x})_m)_{(\mubar,s)\in C_d}$$
and
$$G'_d=(G^\bullet_{\etabar}(\lambda;0;x)_m)_{\etabar\in B_d}$$
to be two column vectors indexed by $(\mubar,s)$ and $\etabar$ respectively. Let $$\tilde{\Phi}_d(\lambda;x)=(\tilde{\Phi}_d^{(\mubar,s),\etabar}(\lambda;x))_{(\mubar,s)\in C_d,\etabar\in B_d}$$
be a matrix indexed by $(\mubar,s)$ and $\etabar$, where
$$\tilde{\Phi}_d^{(\mubar,s),\etabar}(\lambda;x)=\left\{\begin{array}{ll}0,&\textrm{if}|\etabar|>|\mubar|\\
Z_{w^s(-\etabar)}
\tilde{\Phi}^\bullet_{-w^s(-\etabar),\mubar}(-\frac{\sqrt{-1}s}{m}\lambda;\tilde{x})_m,&\textrm{if}|\etabar|=|\mubar|\\
\sum_{|\xibar|=|\mubar|-|\etabar|,l(\xibar)=l'(\xibar)}G^\bullet_{\xibar}(\lambda;0;x)_mZ_{w^s(-(\xibar\sqcup\etabar))}
\tilde{\Phi}^\bullet_{-w^s(-(\xibar\sqcup\etabar)),\mubar}(-\frac{\sqrt{-1}s}{m}\lambda;\tilde{x})_m
&\textrm{if}|\etabar|<|\mubar|\end{array} \right.$$
Then we have the following theorem
\begin{Theorem}
$\tilde{\Phi}_d(\lambda;x)$ is invertible and
\begin{eqnarray*}
G'_d=\tilde{\Phi}_d^{-1}(\lambda;x)\beta_d
\end{eqnarray*}
\end{Theorem}

Theorem 2 in fact determines $G^\bullet_{\mubar}(\lambda;0;x)_m$ for any $\mubar$ because of the following two results that we will show in section 4:
\begin{enumerate}
\item For any $\mubar$ with $k_1=\cdots=k_{l(\mubar)}=0$ we have
$$
G^\bullet_{\mubar}(\lambda;0;x)_m=\frac{1}{|\Aut(\mubar)|}
\prod_{i=1}^{l(\mubar)}\left(-\frac{\sqrt{-1}^{\mu_i+1}}{2m\mu_i\sin(\frac{\mu_i\lambda}{2})}\right).
$$

\item For any $\nubar$, let $\xibar=\{(\nu_i,l_i)|i\in A'(\nubar)\}$ and $\etabar=\{(\nu_i,l_i)|i\in A''(\nubar)\}$. Then
$$
G^\bullet_{\nubar}(\lambda;0;x)_m=G^\bullet_{\xibar}(\lambda;0;x)_mG^\bullet_{\etabar}(\lambda;0;x)_m
.$$
\end{enumerate}
Therefore, the only nontrivial vertices are those $G^\bullet_{\mubar}(\lambda;0;x)_2$ with $k_1,\cdots,k_{l(\mubar)}$ nontrivial and Theorem 2 calculates all of them.

We will also calculate the degree 1 and degree 2 $\bZ_2$-vertices more explicitly in section 5. Then we will use these results to calculate the prediction of some $\bZ_2$-Hodge integrals in \cite{Ros}. These results can be viewed as an evidence for the conjecture of the orbifold GW/DT correspondence in \cite{Ros}.

\subsection{The Gromov-Witten Invariants of the Local $\cB\bZ_m$ Gerbe}
Let $\cX$ be the global quotient of the resolved conifold Tot$(\cO(-1)\oplus\cO(-1)\to \bP^1)$ by $\bZ_m$ acting fiberwise by $\xi_m$ and $\xi_m^{-1}$ respectively, where $\xi_m=e^{\frac{2\pi \sqrt{-1}}{m}}$. Let $\cY_0=\bP^{1}\times \cB\bZ_m$ be the trivial $\bZ_m$ gerbe over $\bP^1$. $\cY_0$ can be viewed as $\bP^1$ with root construction \cite{Cad} of order $m$ for $\cO_{\bP^1}$. Then $\cX$ can be identified with Tot$(L_0\otimes\cO_{\cY_0}(-1)\oplus L_0^{-1}\otimes\cO_{\cY_0}(-1)\to \cY_0)$ where $L_0$ is the tautological bundle on $\cY_0$.

Define $C^\bullet_{\chi,d,\gamma}$ to be
$$C^\bullet_{\chi,d,\gamma}=\int_{[\Mbar^\bullet_{\chi,\gamma}(\cY_0,d)]^\vir}
e(R^1\pi_*F^*(L_0\otimes\cO_{\cY_0}(-1)\oplus L_0^{-1}\otimes\cO_{\cY_0}(-1)))$$
where $\Mbar^\bullet_{\chi,\gamma}(\cY_0,d)$ is the moduli space of degree $d$ stable maps from a possibly disconnected curve with Euler characteristic $\chi$ and with monodromies $\gamma$ around marked points to $\cY_0$,
$$\pi:\cU\to \Mbar^\bullet_{\chi, \gamma}(\cY_0, d)$$
is the universal domain curve and
$$F:\cU\to \cY_0$$
is the evaluation map. Let
$$C^\bullet_{d}(\lambda;x)=\sum_{\chi,\gamma}\lambda^{-\chi+l(\gamma)}C^\bullet_{\chi,d,\gamma}\frac{x_\gamma}{\gamma!}$$
Then we have the following theorem which gives the Gromov-Witten invariants of the local $\cB\bZ_m$ gerbe
\begin{Theorem}
\begin{eqnarray*}
C^\bullet_{d}(\lambda;x)=\sum_{|\mubar|=d}(-1)^{d-l'(\mubar)}G^\bullet_{\mubar}(\lambda;0;x)_m
Z_{\mubar}G^\bullet_{-\mubar}(\lambda;0;x)_m
\end{eqnarray*}
\end{Theorem}

\subsection{Acknowledgments}
I wish to express my deepest thanks to my advisor Chiu-Chu Melissa Liu. When I encountered difficulties, she discussed problems with me patiently and helped me to find many important references. Her guidance has been always enlightening for me. Her papers \cite{Li-Liu-Liu-Zhou} \cite{Liu1} \cite{Liu-Liu-Zhou1} \cite{Liu-Liu-Zhou2} guided me through the whole process of this work. This paper could not be possible without my advisor Chiu-Chu Melissa Liu. I also wish to thank Dustin Ross for his helpful communications which are important for this paper.

\section{Moduli Spaces of Relative Stable Morphisms} \label{moduli}
\subsection{Moduli spaces}
Fix an integer $m\geq 1$. For any integer $s$ with $0\leq s\leq m-1$, let $\cY_s$ be the root construction \cite{Cad} of order $m$ for $\cO_{\bP^1}(-s)$. Then $\cY_0=\bP^{1}\times \cB\bZ_m$ is the trivial gerbe over $\bP^1$ and $\cY_s$ is the nontrivial gerbe over $\bP^1$ for $1\leq s\leq m-1$.
For any integer $a>0$ and $0\leq s\leq m-1$, let
\begin{eqnarray*}
\cY_s[a]&=&\cY_{s}\cup\cY_{0(1)}\cup\cdots\cup\cY_{0(a)}\cong \cY_{s}\cup((\bP^1_{(1)}\cup\cdots\cup\bP^1_{(a)})\times \cB\bZ_m)
\end{eqnarray*}
be the union of $\cY_s$ and $a$ copies of $\cY_0$, where $\cY_s$ is glued to $\cY_{0(1)}$ at $p_1^{(0)}$ and $\cY_{0(l)}$ is glued to $\cY_{0(l+1)}$ at $p_1^{(l)}$ for $1\leq l \leq a-1$. We call the irreducible component $\cY_{s}$ the root component and the other irreducible components the bubble components. A point $p_1^{(a)}\neq p_1^{(a-1)}$ is fixed on $\cY_{0(a)}$. Denote by $\pi_s[a]: \cY_s[a] \to \cY_s$ the map which
is identity on the root component and contracts all the bubble components
to $p_1^{(0)}$. Let
$$
\cY_0(a)=\cY_{0(1)}\cup\cdots\cup\cY_{0(a)}
$$
denote the union of bubble components of $\cY_s[a]$.

Let $\gamma=(\gamma_{1}, \cdots, \gamma_{n})$ be the vector of integers
$$1\leq \gamma_{i}\leq m-1$$
defining nontrivial elements $\gamma_{i}\in \bZ_{m}$. Let
$$
\mubar=\{(\mu_1,k_1),\cdots , (\mu_{l(\mu)}, k_{l(\mu)})\}
$$
be a $\bZ_m$-weighted partition of an integer $d>0$. Here, $\mu$ is a partition of $d$ with parts $\mu_i$ and $k_i\in \bZ_m$. Let
$$\{1,\cdots, l(\mubar)\}=A'(\mubar)\sqcup A''(\mubar)$$
such that $k_i=0$ if and only if $i\in A'(\mubar)$. Let $l'(\mubar)=|A'(\mubar)|$ and $l''(\mubar)=|A''(\mubar)|$. We set a convention for $\mubar$ as follows: If $i<j$, then $\mu_i\geq \mu_j$; if in addition $\mu_i=\mu_j$, let $c=\Gcd(\mu_i,m)$, $\bar{k}_i=k_i$(mod $c$) and $\bar{k}_j=k_j$(mod $c$), then $\bar{k}_j\geq \bar{k}_i$ where $\bar{k}_i,\bar{k}_j$ are viewed as elements in $\{0,\cdots ,c-1\}$; if in addition $\bar{k}_j= \bar{k}_i$, then $k_j\geq k_i$ where $k_i,k_j$ are viewed as elements in $\{0,\cdots ,m-1\}$. We also set the convention for $\gamma$: If $i<j$, then $\gamma_j\geq \gamma_i$ where $\gamma_1,\cdots, \gamma_n$ are viewed as elements in $\{1,\cdots ,m-1\}$.

For any $0\leq s\leq m-1$, let $\Mbar_{g, \gamma}(\cY_s, \mubar)$ be the moduli space of relative maps to $(\cY_s, \infty)$. Then a point in $\Mbar_{g, \gamma}(\cY_s, \mubar)$ is of the form
$$
f: (C, x_1,\cdots , x_n, y_1, \cdots , y_{l(\mubar)})\to (\cY_s[a],p_1^{(a)})
$$
such that
$$
f^{-1}(p_1^{(a)})=\sum_{i=1}^{l(\mubar)}\mu_i\frac{m}{\Gcd(m,k_i)}y_i
$$
as Cartier divisors and the monodromies around $y_i$ and $x_j$ are given by $k_i$ and $\gamma_j$ respectively. For a more general discussion of moduli spaces of relative stable morphisms to orbifolds, see \cite{Abr-Fan}.

We will also consider the disconnected version  $\Mbar^\bullet_{\chi, \gamma}(\cY_s, \mubar)$, where the domain curve $C$ is allowed to be disconnected with $2(h^0(\cO_{c})-h^0(\cO_{c}))=\chi$. Similarly, if we specify ramification types $\nubar,\mubar$ over $0,\infty\in\cY_s$, we can define the corresponding moduli spaces $\Mbar_{g, \gamma}(\cY_s,\nubar, \mubar)$ and $\Mbar^\bullet_{\chi,\gamma}(\cY_s,\nubar, \mubar)$ of relative stable maps.

\subsection{Torus action}
Consider the $\bC^*$-action
$$t \cdot [z^0:z^1] = [tz^0: z^1]$$
on $\bP^1$. This action lifts canonically on $\cY_s$ for any $0\leq s\leq m-1$. This induces an action on $\cY_s[a]$ with trivial actions on the bubble components. These in turn induce actions on the moduli spaces $\Mbar_{g, \gamma}(\cY_s, \mubar), \Mbar^\bullet_{\chi, \gamma}(\cY_s, \mubar), \Mbar_{g, \gamma}(\cY_s,\nubar, \mubar)$ and $\Mbar^\bullet_{\chi, \gamma}(\cY_s,\nubar, \mubar)$. Define the quotient space $\Mbar^\bullet_{\chi, \gamma}(\cY_s,\nubar, \mubar)//\bC^*$ to be
$$\Mbar^\bullet_{\chi, \gamma}(\cY_s,\nubar, \mubar)//\bC^*=(\Mbar^\bullet_{\chi, \gamma}(\cY_s,\nubar, \mubar)\setminus \Mbar^\bullet_{\chi, \gamma}(\cY_s,\nubar, \mubar)^{\bC^*})/\bC^*$$

\subsection{The obstruction bundle}
For any $0\leq s\leq m-1$, let
$$\pi_s:\cU_s\to \Mbar^\bullet_{\chi, \gamma}(\cY_s, \mubar)$$
be the universal domain curve and let $\cT_s$ be the universal target. Then there is an evaluation map
$$F_s:\cU_s\to \cT_s$$
and a contraction map
$$\tpi_s: \cT_s\to \cY_s.$$
Let $L_s=\sqrt[m]{\cO_{\bP^1}(-s)}$ be the tautological line bundle on $\cY_s$ corresponding to the root construction. Then over each point of $\cY_s$, the isotropy group $\bZ_m$ acts on the fiber of $L_s$ by multiplication by $e^{2\pi i/m}$.

Let $\cD_s\subset \cU_s$ be the divisor corresponding to the $l(\mubar)$ marked points $\{ y_1, \cdots , y_{l(\mubar)}\}$ and let $\cD_s'\subset\cU_s$ be the divisor corresponding to those marked points in $\{ y_1, \cdots , y_{l(\mubar)}\}$ which have trivial monodromies i.e. those marked points $y_i$ with $i\in A'(\mubar)$. Define
\begin{eqnarray*}
V^0_D&=&R^1(\pi_0)_*(\tilde{F_0}^*L_0(-\mathcal{D}_0') )\\
V^0_{D_d}&=&R^1(\pi_0)_* \tilde{F_0}^*(L_0^{-1}\otimes\cO_{\cY_0}(-p_{0})),
\end{eqnarray*}
where $\tilde{F}_0=\tilde\pi_0\circ F_0:\mathcal{U}_0\to \cY_0$ and $p_0=0,p_1=\infty\in\cY_0$. The fibers of $V^0_D$ and $V^0_{D_d}$ at
$$
\left[ f:(C,x_1,\ldots,x_{n}, y_1, \cdots , y_{l(\mubar)})\to \cY_0[a]\ \right]\in \Mbar^\bullet_{\chi, \gamma}(\cY_0, \mubar)
$$
are $H^1(C, \tilde{f}_0^*L_0(-D_0'))$ and $H^1(C, \tilde{f}_0^*(L_0^{-1}\otimes\cO_{\cY_0}(-p_0)))$,
respectively, where $D_0'=\sum_{i\in A'(\mubar)}y_i$, and $\tilde{f}_0=\pi_0[a]\circ f_0$. Note that
$$H^0(C, \tilde{f}_0^*L_0(-D_0'))=H^0(C, \tilde{f}_0^*(L_0^{-1}\otimes\cO_{\cY_0}(-p_0)))=0,$$
so $V^0_D$ and $V^0_{D_d}$ are vector bundles of ranks $l'(\mubar)-\frac{\chi}{2}+\sum_{i\in A''(\mubar)}\frac{k_i}{m}+\sum_{j=1}^{n}\frac{\gamma_j}{m}$ and $d-\frac{\chi}{2}+\sum_{i\in A''(\mubar)}\frac{m-k_i}{m}+\sum_{j=1}^{n}\frac{m-\gamma_j}{m}$, respectively. The obstruction bundle
$$V^0=V^0_D\oplus V^0_{D_{d}}$$
has rank $-\chi+d+l(\mubar)+n$ which is equal to the virtual dimension of $\Mbar^\bullet_{\chi, \gamma}(\cY_0, \mubar)$.

For $1\leq s\leq m-1$, define
\begin{eqnarray*}
V^s_D&=&R^1(\pi_s)_*(\tilde{F_s}^*L_s )\\
V^s_{D_d}&=&R^1(\pi_s)_* \tilde{F_s}^*(L_s^{-1}\otimes\cO_{\cY_s}(-p_{0})),
\end{eqnarray*}
where $\tilde{F}_s=\tilde\pi_s\circ F_s:\mathcal{U}_s\to \cY_s$ and $p_0=0,p_1=\infty\in\cY_s$. The fibers of $V^s_D$ and $V^s_{D_d}$ at
$$
\left[ f:(C,x_1,\ldots,x_{n}, y_1, \cdots , y_{l(\mubar)})\to \cY_s[a]\ \right]\in \Mbar^\bullet_{\chi, \gamma}(\cY_s, \mubar)
$$
are $H^1(C, \tilde{f}_s^*L_s)$ and $H^1(C, \tilde{f}_s^*(L_s^{-1}\otimes\cO_{\cY_s}(-p_0)))$,
respectively, where $\tilde{f}_s=\pi_s[a]\circ f_s$. Note that
$$H^0(C, \tilde{f}_s^*L_s)=H^0(C, \tilde{f}_s^*(L_s^{-1}\otimes\cO_{\cY_s}(-p_0)))=0,$$
because $\deg\tilde{f}_s^*L_s$ and $\deg\tilde{f}_s^*(L_s^{-1}\otimes\cO_{\cY_s}(-p_0))$ are less than zero. So $V^s_D$ and $V^s_{D_d}$ are vector bundles of ranks $\frac{sd}{m}-\frac{\chi}{2}+\sum_{i\in A''(\mubar)}\frac{k_i}{m}+\sum_{j=1}^{n}\frac{\gamma_j}{m}$ and $\frac{(m-s)d}{m}-\frac{\chi}{2}+\sum_{i\in A''(\mubar)}\frac{m-k_i}{m}+\sum_{j=1}^{n}\frac{m-\gamma_j}{m}$, respectively. The obstruction bundle
$$V^s=V^s_D\oplus V^s_{D_{d}}$$
has rank $-\chi+d+l''(\mubar)+n$ which is less than the virtual dimension of $\Mbar^\bullet_{\chi, \gamma}(\cY_s, \mubar)$ if $k_i=0$ for some $1\leq i\leq l(\mubar)$.

We lift the $\bC^*$-action to the obstruction bundle $V^s$ for $0\leq s\leq m-1$. It suffices to lift the $\bC^*$-action on $\cY_s$ to the line bundles $L_s$ and $L_s^{-1}\otimes\cO_{\cY_s}(-p_0)$. Let the weights of the $\bC^*$-action on $L_0^{-1}\otimes\cO_{\cY_0}(-p_0)$ at $p_0$ and $p_1$ be $-\tau-1$ and $-\tau$, respectively and let the weights of the $\bC^*$-action on $L_0$ at $p_0$ and $p_1$ be $\tau$ and $\tau$, respectively, where $\tau\in \frac{1}{m}\bZ$. For $1\leq s\leq m-1$, let the weights of the $\bC^*$-action on $L_s^{-1}\otimes\cO_{\cY_s}(-p_0)$ at $p_0$ and $p_1$ be $-1$ and $-\frac{s}{m}$, respectively and let the weights of the $\bC^*$-action on $L_s$ at $p_0$ and $p_1$ be $0$ and $\frac{s}{m}$, respectively.

For $0\leq s\leq m-1$, let
\begin{eqnarray*}
\xn{K}=\frac{1}{\am}\int_{[\Mbar^\bullet_{\chi, \gamma}(\cY_s, \mubar)]^{\vir}}e(V^s)
\end{eqnarray*}
Then $\xn{K}$ is a topological invariant and we have
$$\xn{K}=0$$
when $1\leq s\leq m-1$ and $k_i=0$ for some $1\leq i\leq l(\mubar)$. We will calculate $\xn{K}$ in section 4 by virtual localization.

\section{Orbifold Rubber Calculus}

\subsection{The orbifold rubber integral $H^\bullet_{\chi,\gamma}(\mubar,\nubar)_m$ and wreath Hurwitz numbers}
Similar to the case of $\Mbar^\bullet_{\chi, \gamma}(\cY_0, \mubar)$, a point $[f]\in \Mbar^\bullet_{\chi, \gamma}(\cY_0,\nubar, \mubar)$ has target of the form $\cY_0[a_{0},a_1]$, where $\cY_0[a_{0},a_1]$ is obtained by attaching $\cY_0(a_0)$ and $\cY_0(a_1)$ to $\cY_0$ at 0 and $\infty$ respectively. The distinguished points on $\cY_0[a_{0},a_1]$ are $q^0_{a_0}$ and $q^1_{a_1}$. Let $\pi: \cY_0[a_{0},a_1]\to \cY_0$ be the contraction to the root component. We always assume the ramification type over $q^1_{a_1}$ is $\mubar$ and the ramification type over $q^0_{a_0}$ is $\nubar$.

We will study the following kind of orbifold rubber integrals
$$H^\bullet_{\chi,\gamma}(\mubar,\nubar)_m=\frac{(-\chi+l(\mubar)+l(\nubar))!}{\amm}\int_{[\Mbar^\bullet_{\chi, \gamma}(\cY_0, \mubar,\nubar)//\bC^*]^{\vir}}(\psi^{0})^{-\chi+l(\mubar)+l(\nubar)+l(\gamma)-1}$$
where $\psi^{0}$ is the target $\psi$ class, the first Chern class of the line bundle $\bL_{0}$ over $\Mbar^\bullet_{\chi, \gamma}(\cY_0, \mubar,\nubar)//\bC^*$ whose fiber at
$$[f:C\to \cY_0[a_{0},a_1]]$$
is the cotangent line $T^*_{q^0_{a_0}}\cY_0[a_{0},a_1]$.

When $\gamma=\emptyset$, $H^\bullet_{\chi,\emptyset}(\mubar,\nubar)_m$ is just the disconnected wreath Hurwitz number \cite{Joh} \cite{Zhang-Zhou}. Note that $\chi\leq \textrm{min}\{2l(\mubar),2l(\nubar)\}$ and if $\gamma=\emptyset$, the equality holds if and only if $\nubar=-\mubar$, where $-\mubar$ is defined to be
$$-\mubar:=\{(\mu_1,-k_1),\cdots , (\mu_{l(\mu)}, -k_{l(\mu)})\}$$
In this case vir.dim$\Mbar^\bullet_{\chi, \gamma}(\cY_0, \mubar,\nubar)//\bC^*=-1$ and we set the convention that
$$H^\bullet_{2l(\mubar),\emptyset}(\mubar,-\mubar)_m=\frac{1}{Z_{\mubar}}$$
where $Z_{\mubar}=|\Aut(\mubar)|m^{l(\mubar)}\prod_{i=1}^{l(\mubar)}\mu_i$. The same convention is used in the study of wreath Hurwitz numbers since the Burnside formula for wreath Hurwitz numbers extends naturally to this boundary case. See \cite{Joh} \cite{Zhang-Zhou} for more details on the Burnside formula and other combinatorial expressions of the wreath Hurwitz numbers.

We define generating functions of $H^\bullet_{\chi,\gamma}(\mubar,\nubar)_m$:
\begin{eqnarray*}
\Phi^\bullet_{\mubar,\nubar,\gamma}(\lambda)_m&=&\sum_{\chi\in 2\bZ,\chi\leq\min\{2l(\mubar),2l(\nubar)\}}\frac{\lambda^{-\chi+l(\mubar)+l(\nubar)+l(\gamma)}}
{(-\chi+l(\mubar)+l(\nubar))!}H^\bullet_{\chi,\gamma}(\mubar,\nubar)_m\\
\Phi^\bullet(\lambda;p^+,p^-,x)_m&=&\sum_{\mubar,\nubar,\gamma}\Phi^\bullet_{\mubar,\nubar,\gamma}(\lambda)_mp^+_{\mubar}
p^-_{\nubar}\frac{x_\gamma}{\gamma!}=\sum_{\mubar,\nubar}\Phi^\bullet_{\mubar,\nubar}(\lambda;x)_mp^+_{\mubar}
p^-_{\nubar}
\end{eqnarray*}
where $p^+=(p^+_{(i,j)})_{i\in\bZ_+,j\in\{0,\cdots,m-1\}},p^-=(p^-_{(i,j)})_{i\in\bZ_+,j\in\{0,\cdots,m-1\}},x=(x_1,\cdots,x_{m-1})$ are formal variables, $p^+_{\mubar}=p^+_{(\mu_1,k_1)}\cdots p^+_{(\mu_{l(\mubar)},k_{l(\mubar)})}, p^-_{\nubar}=p^-_{(\nu_1,l_1)}\cdots p^-_{(\nu_{l(\nubar)},l_{l(\nubar)})},x_{\gamma}=x_{\gamma_1}\cdots x_{\gamma_n}$ and we use the more intuitive symbol $\gamma!$ to denote $|\Aut(\gamma)|$. We will also consider the connected orbifold rubber integral
$$H^\circ_{g,\gamma}(\mubar,\nubar)_m=\frac{(2g-2+l(\mubar)+l(\nubar))!}{\amm}\int_{[\Mbar_{g, \gamma}(\cY_0, \mubar,\nubar)//\bC^*]^{\vir}}(\psi^{0})^{2g-2+l(\mubar)+l(\nubar)+l(\gamma)-1}$$
and the corresponding generating functions
\begin{eqnarray*}
\Phi^\circ_{\mubar,\nubar,\gamma}(\lambda)_m&=&\sum_{g=0}^{\infty}\frac{\lambda^{2g-2+l(\mubar)+l(\nubar)+l(\gamma)}}
{(2g-2+l(\mubar)+l(\nubar))!}H^\circ_{g,\gamma}(\mubar,\nubar)_m\\
\Phi^\circ(\lambda;p^+,p^-,x)_m&=&\sum_{\mubar\neq\emptyset,\nubar\neq\emptyset,\gamma}
\Phi^\circ_{\mubar,\nubar,\gamma}(\lambda)_mp^+_{\mubar}
p^-_{\nubar}\frac{x_\gamma}{\gamma!}
=\sum_{\mubar\neq\emptyset,\nubar\neq\emptyset}\Phi^\circ_{\mubar,\nubar}(\lambda;x)_mp^+_{\mubar}
p^-_{\nubar}
\end{eqnarray*}
Then we have the following relation:
\begin{equation}\label{eqn:expphim}
\Phi^\bu(\lam;p^+,p^-,x)_m =\exp(\Phi^\circ(\lam;p^+,p^-,x)_m).
\end{equation}

Notice that although we take $\gamma$ to be a vector of nontrivial elements in $\bZ_m$, the above construction of rubber integrals and their generating functions works for all $\gamma$. In particular, we can apply our construction to the non-orbifold case. So we define
\begin{eqnarray*}
H^\bullet_{\chi,n}(\mu,\nu):&=&\frac{(-\chi+l(\mu)+l(\nu)+n)!}{|\Aut(\nu)||\Aut(\mu)|}\int_{[\Mbar^\bullet_{\chi, n}(\bP^1, \mu,\nu)//\bC^*]^{\vir}}(\psi^{0})^{-\chi+l(\mu)+l(\nu)+n-1}\\
H^\circ_{g,n}(\mu,\nu):&=&\frac{(2g-2+l(\mu)+l(\nu)+n)!}{|\Aut(\nu)||\Aut(\mu)|}\int_{[\Mbar_{g, n}(\bP^1, \mu,\nu)//\bC^*]^{\vir}}(\psi^{0})^{2g-2+l(\mu)+l(\nu)+n-1}\\
\end{eqnarray*}
We also define their generating functions to be
\begin{eqnarray*}
\Phi^\bullet_{\mu,\nu,n}(\lambda)&=&\sum_{\chi\in 2\bZ,\chi\leq\min\{2l(\mu),2l(\nu)\}}\frac{\lambda^{-\chi+l(\mu)+l(\nu)+n}}
{(-\chi+l(\mu)+l(\nu)+n)!}H^\bullet_{\chi,n}(\mu,\nu)\\
\Phi^\bullet(\lambda;p^+,p^-,x)&=&\sum_{\mu,\nu,n}\Phi^\bullet_{\mu,\nu,n}(\lambda)p^+_{\mu}
p^-_{\nu}\frac{x^n}{n!}=\sum_{\mu,\nu}\Phi^\bullet_{\mu,\nu}(\lambda;x)p^+_{\mu}
p^-_{\nu}\\
\Phi^\circ_{\mu,\nu,n}(\lambda)&=&\sum_{g=0}^{\infty}\frac{\lambda^{2g-2+l(\mu)+l(\nu)+n}}
{(2g-2+l(\mu)+l(\nu)+n)!}H^\circ_{g,n}(\mu,\nu)\\
\Phi^\circ(\lambda;p^+,p^-,x)&=&\sum_{\mu\neq\emptyset,\nu\neq\emptyset,n}
\Phi^\circ_{\mu,\nu,n}(\lambda)p^+_{\mu}
p^-_{\nu}\frac{x^n}{n!}
=\sum_{\mu\neq\emptyset,\nu\neq\emptyset}\Phi^\circ_{\mu,\nu}(\lambda;x)p^+_{\mu}
p^-_{\nu}
\end{eqnarray*}
Then we also have the relation
\begin{equation}\label{eqn:expphi1}
\Phi^\bu(\lam;p^+,p^-,x) =\exp(\Phi^\circ(\lam;p^+,p^-,x)).
\end{equation}

\subsection{Calculation of $H^\bullet_{\chi,\gamma}(\mubar,\nubar)_m$}
In this subsection, we will first give a geometric interpretation of $H^\bullet_{\chi,\gamma}(\mubar,\nubar)_m$. Then we will calculate our orbifold rubber integral $H^\bullet_{\chi,\gamma}(\mubar,\nubar)_m$ in two different ways: one is to express $H^\bullet_{\chi,\gamma}(\mubar,\nubar)_m$ in terms of the usual (non-orbifold) double Hurwitz numbers and the other is to give a combinatorial expression of $H^\bullet_{\chi,\gamma}(\mubar,\nubar)_m$ using the representation theory of the wreath product. In what follows, we assume $\gamma$ to be a vector of nontrivial elements in $\bZ_m$.
\subsubsection{A geometric interpretation of $H^\bullet_{\chi,\gamma}(\mubar,\nubar)_m$}
\begin{definition}
For given $g,\gamma,\mubar,\nubar$, we fix $2g-2+l(\mubar)+l(\nubar)+l(\gamma)$ different points on $\bP^1\setminus \{0,\infty\}$ and define $\hat{H}^\circ_{g,\gamma}(\mubar,\nubar)_m$ to be the count of degree $md$ ($d=|\mubar|=|\nubar|$) covers $f:\tilde{C}\to \bP^1$, with monodromy in the wreath product $\bZ_m\wr S_d$ (see \cite{Mac}), with prescribed monodromy: the monodromy over 0 and $\infty$ must be $\mubar$ and $\nubar$ respectively, the monodromy over each of the $2g-2+l(\mubar)+l(\nubar)$ (fixed) points must be $\{(2,0),(1,0),\cdots,(1,0)\}$, the monodromy over the (fixed) point corresponding to $\gamma_i$ must be $\{(1,\gamma_i),(1,0),\cdots,(1,0)\}$ and $\tilde{C}/\bZ_m$ is a connected genus $g$ twisted curve. If we do not require $\tilde{C}/\bZ_m$ to be connected and require the Euler characteristic of $\tilde{C}/\bZ_m$ to be $\chi$, then the corresponding number of covers is denoted by $\hat{H}^\bullet_{\chi,\gamma}(\mubar,\nubar)_m$.
\end{definition}

For $n=l(\gamma)$, consider the canonical map
$$\rho: \Mbar_{g, \gamma}(\cY_0, \mubar,\nubar)\to \Mbar_{g, n}(\bP^1, \mu,\nu)$$
which forgets the orbifold structure. Then we have the following lemma which is completely similar to lemma 6 in \cite{JPT}.
\begin{lemma} Consider $\rho|_{\cM_{g, \gamma}(\cY_0, \mubar,\nubar)//\bC^*}:\cM_{g, \gamma}(\cY_0, \mubar,\nubar)\to \cM_{g, n}(\bP^1, \mu,\nu)$, we have
$$
\deg(\rho|_{\cM_{g, \gamma}(\cY_0, \mubar,\nubar)})=\left\{\begin{array}{ll}0,
&\sum_{j=1}^{n}\gamma_{j}+\sum_{i=1}^{l(\mubar)}k_i+\sum_{q=1}^{l(\nubar)}l_q\neq 0\\
m^{2g-1},&\sum_{j=1}^{n}\gamma_{j}+\sum_{i=1}^{l(\mubar)}k_i+\sum_{q=1}^{l(\nubar)}l_q=0\end{array}\right.
$$
where $\gamma=(\gamma_{1}, \cdots, \gamma_{n}), \mubar=\{(\mu_1,k_1),\cdots , (\mu_{l(\mu)}, k_{l(\mu)})\},\nubar=\{(\nu_1,l_1),\cdots , (\nu_{l(\nu)}, l_{l(\nu)})\}$.
\end{lemma}
\begin{proof}
Let $ \left[ f:(C,x_1,\ldots,x_{n}, y_1, \cdots , y_{l(\mubar)},z_1,\cdots,z_{l(\nubar)})\to \cY_m[a_{0},a_1] \right]$ be a point in $\cM_{g, \gamma}(\cY_0, \mubar,\nubar)$ and let $(C',x_1',\ldots,x_{n}', y_1', \cdots , y_{l(\mubar)}',z_1',\cdots,z_{l(\nubar)}')$ be the coarse curve of $C$. The map $f$ together with the projection map $\cY_m[a_{0},a_1]\to \bP^1[a_0,a_1]$ induce a map $f':(C',x_1',\ldots,x_{n}', y_1', \cdots , y_{l(\mubar)}',z_1',\cdots,z_{l(\nubar)}')\to \bP^1[a_{0},a_1]$. Then we have $\rho([f])=[f']$. Conversely, given $[f']$, the map from $C$ to $\bP^1$ is given by the composition of $f'$ and the canonical map $C\to C'$. So the preimage of $[f']$ is parameterized by the maps $C\to \cB\bZ_m$ with given monodromies $\gamma,k,l$ at the corresponding marked points. Therefore, if $\sum_{j=1}^{n}\gamma_{j}+\sum_{i=1}^{l(\mubar)}k_i+\sum_{q=1}^{l(\nubar)}l_q\neq 0$, then there are $m^{2g}$ points in a fiber of $\rho$ corresponding to the monodromies around the $2g$ noncontractible loops on $C'$. Since $\bZ_m$ is abelian, a $\bZ_m$-cover has automorphism group $\bZ_m$. So the degree of $\rho$ is $m^{2g-1}$.
\end{proof}

Recall that we have a branch morphism
$$Br:\Mbar_{g, n}(\bP^1, \mu,\nu)\to\textrm{Sym}^r\bP^1\cong \bP^r$$
where $r=2g-2+l(\mu)+l(\nu)$, and evaluation maps
$$ev_i: \Mbar_{g, n}(\bP^1, \mu,\nu)\to\bP^1$$
for $i=1,\cdots,n$. The usual nonsingularity and Bertini arguments \cite{Fan-Pan} show that
\begin{equation}\label{eqn:orbibranch}
\hat{H}^\circ_{g,\gamma}(\mubar,\nubar)_m=\frac{1}{|\Aut(\mubar)||\Aut(\nubar)|}\int_{[\Mbar_{g, \gamma}(\cY_0, \mubar,\nubar)]^\vir}(\rho\circ Br)^*(pt)\cdot (\rho\circ ev_1)^*(pt)\cdots(\rho\circ ev_n)^*(pt)
\end{equation}
Localization calculations similar to those in \cite{Liu-Liu-Zhou2} show that
$$\hat{H}^\circ_{g,\gamma}(\mubar,\nubar)_m= \frac{(2g-2+l(\mubar)+l(\nubar))!}{\amm}\int_{[\Mbar_{g, \gamma}(\cY_0, \mubar,\nubar)//\bC^*]^{\vir}}(\psi^{0})^{2g-2+l(\mubar)+l(\nubar)+l(\gamma)-1}.$$
In other words, we have
$$\hat{H}^\circ_{g,\gamma}(\mubar,\nubar)_m=H^\circ_{g,\gamma}(\mubar,\nubar)_m.$$
Similarly, we also have $\hat{H}^\bullet_{\chi,\gamma}(\mubar,\nubar)_m=H^\bullet_{\chi,\gamma}(\mubar,\nubar)_m$.

Using the same localization calculations, the following identity holds
\begin{equation}\label{eqn:branch}
H^\circ_{g,n}(\mu,\nu)=\frac{1}{|\Aut(\mu)||\Aut(\nu)|}\int_{[\Mbar_{g, n}(\bP^1, \mu,\nu)]^\vir} Br^*(pt)\cdot  ev_1^*(pt)\cdots ev_n^*(pt)
\end{equation}

Lemma 3.1 together with (\ref{eqn:orbibranch}) (\ref{eqn:branch}) and the nonsingularity and Bertini arguments show that
\begin{equation}\label{eqn:deg}
H^\circ_{g,\gamma}(\mubar,\nubar)_m=\hat{H}^\circ_{g,\gamma}(\mubar,\nubar)_m=\frac{|\Aut(\nu)||\Aut(\mu)|}{\amm}\delta_{0,\langle \frac{\sum_{j=1}^{n}\gamma_{j}+\sum_{i=1}^{l(\mubar)}k_i+\sum_{q=1}^{l(\nubar)}l_q}{m}\rangle}m^{2g-1}H^\circ_{g,n}(\mu,\nu)
\end{equation}

By the divisor equation, we have
\begin{equation}\label{eqn:div}
H^\circ_{g,n}(\mu,\nu)=\frac{d^n}{|\Aut(\mu)||\Aut(\nu)|}\int_{[\Mbar_{g, n}(\bP^1, \mu,\nu)]^\vir} Br^*(pt)=d^nH^\circ_{g}(\mu,\nu)
\end{equation}
and hence
\begin{equation}\label{eqn:n}
H^\bullet_{\chi,n}(\mu,\nu)=d^nH^\bullet_{\chi}(\mu,\nu)
\end{equation}

Another way to obtain (\ref{eqn:div}) (\ref{eqn:n}) is to use the degeneration formula. By equation (2.10) in \cite{Oko-Pan1}, we have the following relation
$$H^\bullet_{\chi,n}(\mu,\nu)=\sum_{S\subset \{2,\cdots,n\}}\sum_{\eta }H^\bullet_{\chi,n-1-|S|}(\nu,\eta)z_{\eta}\frac{1}{|\Aut(\mu)||\Aut(\eta)|}\int_{[\Mbar^\bullet_{\chi', |S|+1}(\bP^1, \eta,\mu)]^{\vir}}\omega $$
where $z_{\eta}=|\Aut(\eta)|\eta_1\cdots\eta_{l(\eta)}$, $\chi'$ is chosen to make the second integral nonzero and $\omega$ is the Poincare dual of a point. But the only way to make the second integral nonzero is to set $S=\emptyset, \eta=\mu$ and $\chi'=2l(\mu)$. In this case, we have
$$\frac{1}{|\Aut(\mu)||\Aut(\eta)|}\int_{[\Mbar^\bullet_{\chi', |S|+1}(\bP^1, \eta,\mu)]^{\vir}}\omega=\frac{1}{|\Aut(\eta)|}\sum_{i=1}^{l(\eta)}\frac{\eta_i}{\eta_1\cdots\eta_{l(\eta)}}$$
Therefore we have
$$H^\bullet_{\chi,n}(\mu,\nu)=dH^\bullet_{\chi,n-1}(\mu,\nu)$$
where $d=|\mu|=|\nu|$.
Repeating this process for $n$ times, we obtain (\ref{eqn:n}).

In conclusion, equations (\ref{eqn:expphim}) (\ref{eqn:deg}) (\ref{eqn:div}) completely determine the orbifold rubber integral $H^\bullet_{\chi,\gamma}(\mubar,\nubar)_m$.

\subsubsection{A combinatorial expression of $H^\bullet_{\chi,\gamma}(\mubar,\nubar)_m$}
Let $(\bZ_m)_d$ denote the wreath product $\bZ_m\wr S_d$ (see \cite{Mac}). Any $\bZ_m$-weighted partition $\mubar$ with $|\mubar|=d$ can be viewed as a conjugacy class in $\bZ_m\wr S_d$. By the geometric interpretation of $H^\bullet_{\chi,\gamma}(\mubar,\nubar)_m$, it is easy to show that $H^\bullet_{\chi,\gamma}(\mubar,\nubar)_m$ has the following algebraic definition:
$$H^\bullet_{\chi,\gamma}(\mubar,\nubar)_m=\frac{1}{|(\bZ_m)_d|}|\{(\sigma_0,\sigma_\infty,\sigma_1,\cdots,\sigma_r,\omega_1,
\cdots,\omega_n)\in(\bZ_m)_d^{n+r+2}|\sigma_0\sigma_\infty\cdot\sigma_1\cdots\sigma_r\cdot\omega_1
\cdots\omega_n=1\}$$
such that $\sigma_0$ has type $\mubar$, $\sigma_\infty$ has type $\nubar$, $\sigma_1,\cdots,\sigma_r$ have type $\tau_0$ and $\omega_i$ has type $\rho_{\gamma_i}$, where
$$\tau_0=\{(2,0),(1,0),\cdots,(1,0)\}$$
and
$$\rho_{\gamma_i}=\{(1,\gamma_i),(1,0),\cdots,(1,0)\}.$$

Let $\bC[(\bZ_m)_d]$ be the group algebra associated to $(\bZ_m)_d$ and $Z\bC[(\bZ_m)_d]$ the center of $\bC[(\bZ_m)_d]$. For any $\mubar$, define $C_{\mubar}\in Z\bC[(\bZ_m)_d]$ to be the sum of elements of type $\mubar$. Then by the above algebraic definition of $H^\bullet_{\chi,\gamma}(\mubar,\nubar)_m$, we have
$$H^\bullet_{\chi,\gamma}(\mubar,\nubar)_m=\frac{1}{|(\bZ_m)_d|}[1]C_{\mubar}C_{\nubar}C_{\tau_0}^r\prod_{i=1}^{l(\gamma)}
C_{\rho_{\gamma_i}}$$
where $r=2g-2+l(\mubar)+l(\nubar)$ and for any $x\in \bC[(\bZ_m)_d]$, $[1]x$ means taking the coefficient of the identity element.

The center $Z\bC[(\bZ_m)_d]$ is called the class algebra since it has a basis
$$\{C_{\mubar}\}_{|\mubar|=d}$$
indexed by the conjugacy classes. On the other hand, $Z\bC[(\bZ_m)_d]$ also has a semisimple basis
$$\{E_{\xibar}\}_{|\xibar|=d}$$
indexed by the irreducible representations of $(\bZ_m)_d$. We have
$$E_{\xibar}=\frac{\textrm{dim}\xibar}{|(\bZ_m)_d|}\sum_{|\mubar|=d}X_{\xibar}(\mubar)C_{\mubar}$$
and
$$C_{\mubar}=\sum_{|\xibar|=d}\frac{|C_{\mubar}|X_{\xibar}(\mubar)}{\textrm{dim}\xibar}E_{\xibar}=\sum_{|\xibar|=d}
F_{\xibar}(\mubar)E_{\xibar}$$
where $X_{\xibar}$ and $\textrm{dim}\xibar$ are character and dimension of the irreducible representations of $(\bZ_m)_d$ associated with $\xibar$ respectively and $F_{\xibar}(\mubar)=\frac{|C_{\mubar}|X_{\xibar}(\mubar)}{\textrm{dim}\xibar}$.

Therefore we have,
\begin{eqnarray*}
H^\bullet_{\chi,\gamma}(\mubar,\nubar)_m&=&\frac{1}{|(\bZ_m)_d|}[1]C_{\mubar}C_{\nubar}C_{\tau_0}^r\prod_{i=1}^{l(\gamma)}
C_{\rho_{\gamma_i}}\\
&=&\frac{1}{|(\bZ_m)_d|}[1](\sum_{|\xibar|=d}
F_{\xibar}(\mubar)E_{\xibar})(\sum_{|\xibar|=d}
F_{\xibar}(\nubar)E_{\xibar})(\sum_{|\xibar|=d}
F_{\xibar}(\tau_0)E_{\xibar})^r\prod_{i=1}^{l(\gamma)}(\sum_{|\xibar|=d}
F_{\xibar}(\rho_{\gamma_i})E_{\xibar})\\
&=&\frac{1}{|(\bZ_m)_d|}[1]\sum_{|\xibar|=d}(F_{\xibar}(\mubar)F_{\xibar}(\nubar)F_{\xibar}(\tau_0)^r\prod_{i=1}^{l(\gamma)}
F_{\xibar}(\rho_{\gamma_i}))E_{\xibar}\\
&=&\frac{1}{|(\bZ_m)_d|}\sum_{|\xibar|=d}(F_{\xibar}(\mubar)F_{\xibar}(\nubar)F_{\xibar}(\tau_0)^r\prod_{i=1}^{l(\gamma)}
F_{\xibar}(\rho_{\gamma_i}))\frac{(\textrm{dim}\xibar)^2}{|(\bZ_m)_d|}\\
&=&\sum_{|\xibar|=d}\frac{X_{\xibar}(\mubar)}{Z_{\mubar}}\frac{X_{\xibar}(\nubar)}{Z_{\nubar}}
F_{\xibar}(\tau_0)^r\prod_{i=1}^{l(\gamma)}F_{\xibar}(\rho_{\gamma_i})
\end{eqnarray*}
where in the fourth identity we used the fact that $[1]E_{\xibar}=\frac{(\textrm{dim}\xibar)^2}{|(\bZ_m)_d|}$.

In order to compute $F_{\xibar}(\tau_0)$ and $F_{\xibar}(\rho_{\gamma_i})$, we need to introduce some notations (see \cite{Mac}). For any $\bZ_m$-weighted partition $\mubar$, we can decompose $\mubar$ into the following form
$$\mubar=\overline{\mubar(0)}\sqcup\cdots\sqcup\overline{\mubar(m-1)}$$
where $\overline{\mubar(i)}$ is weighted by the single element $i\in\bZ_m$ i.e. $\overline{\mubar(i)}=\{(\mu(i)_1,i),\cdots,(\mu(i)_{l(\mubar(i))},i)\}$. We denote the underlying partition of $\overline{\mubar(i)}$ by $\mubar(i)$. For any $d\geq 1$ and $k\in \bZ_m$, let $p_d(k)$ be the $d^{th}$ power sum in a sequence of variables $y_k=(y_{(n,k)})_{n\geq 1}$. For any partition $\mu=\{\mu_1,\cdots,\mu_{l(\mu)}\}$, let $p_{\mu}(k)=p_{\mu_1}(k)\cdots p_{\mu_{l(\mu)}}(k)$. For any $\bZ_m$-weighted partition $\mubar$, let
$$P_{\mubar}=\prod_{k\in\bZ_m}p_{\mubar(k)}(k).$$
Then the elements $P_{\mubar}, |\mubar|\geq 1$ form a basis of the following ring
$$\Lambda_m=\bC[p_d(k)]_{d\geq 1,k\in \bZ_m}.$$
We define a bilinear form $\langle\cdot,\cdot\rangle$  on $\Lambda_m$ by setting
$$\langle P_{\mubar},P_{\nubar}\rangle=\delta_{\mubar,\nubar}Z_{\mubar}.$$

On the other hand, for any irreducible character $\alpha$ of $\bZ_m$ and any $d\geq 1$, we can define
$$p_d(\alpha)=\sum_{k\in \bZ_m}\frac{\alpha(k)}{m}p_d(k).$$
Let $\alpha_i(k)=e^{\frac{2ik\pi\sqrt{-1}}{m}},k\in\bZ_m$. Then $\{\alpha_i|0\leq i\leq m-1\}$ is the set of irreducible character of $\bZ_m$ and we have
$$p_d(k)=\sum_{i=0}^{m-1}\alpha_i(-k)p_d(\alpha_i).$$
If we regard $p_d(\alpha)$ as the $d^{th}$ power sum in a sequence of variables $y_{\alpha}=(y_{(n,\alpha)})_{n\geq 1}$, then for any partition $\mu$ we can define the Schur function $s_{\mu}(\alpha)=s_{\mu}(y_{\alpha})$. Therefore for any $\bZ_m$-weighted partition $\mubar$, we can define the Schur function
$$S_{\mubar}=\prod_{i=0}^{m-1}s_{\mubar(i)}(\alpha_i),$$
and $S_{\mubar}, |\mubar|\geq 1$ also form a basis of $\Lambda_m$.

The following proposition can be found in \cite{Mac}
\begin{proposition}
For any $\bZ_m$-weighted partitions $\mubar$ and $\xibar$, one has
$$\langle S_{\xibar},P_{\mubar}\rangle=X_{\xibar}(\mubar).$$
\end{proposition}

In particular, if we define $\epsilon$ to be
$$\epsilon=\{(1,0),\cdots,(1,0)\},$$
then for $|\xibar|=|\epsilon|=d$ we have
$$\textrm{dim}\xibar=X_{\xibar}(\epsilon)=\langle S_{\xibar},P_{\epsilon}\rangle=\langle S_{\xibar},(\sum_{i=0}^{m-1}
p_1(\alpha_i))^d\rangle=d!\prod_{i=0}^{m-1}\frac{\textrm{dim}\xibar(i)}{|\xibar(i)|}.$$

When $m=1$, $F_{\xibar}(\tau_0)$ can be computed in the following proposition (see Example 7 in page 117 of \cite{Mac}):
\begin{proposition}
Let $\xi$ be an ordinary partition of $d\geq 1$. Then
$$F_{\xi}(\tau_0)=\frac{\kappa_{\xi}}{2}$$
where $\kappa_{\xi}=\sum_{i=1}^{l(\xi)}\xi_i(\xi_i-2i+1)$ and $\tau_0=(1^{d-2}2)$.
\end{proposition}

Now let us compute $F_{\xibar}(\tau_0)$ and $F_{\xibar}(\rho_{\gamma_i})$ by proving the following lemma:
\begin{lemma}
\begin{eqnarray*}
F_{\xibar}(\tau_0)&=&\sum_{i=0}^{m-1}\frac{m\kappa_{\xibar(i)}}{2},\\
F_{\xibar}(\rho_{\gamma_i})&=&\sum_{j=0}^{m-1}|\xibar(j)|
e^{\frac{-2\gamma_ij\pi\sqrt{-1}}{m}}
\end{eqnarray*}
\end{lemma}
\begin{proof}
Let $|\xibar|=|\tau_0|=|\rho_{\gamma_i}|=d, d_j=|\xibar(j)|, j=0,\cdots,m-1$. Then we have
\begin{eqnarray*}
X_{\xibar}(\tau_0)&=&\langle S_{\xibar},P_{\tau_0}\rangle=\langle\prod_{i=0}^{m-1}S_{\xibar(i)}(\alpha_i),p_1(0)^{d-2}p_2(0)\rangle\\
&=&\langle\prod_{i=0}^{m-1}S_{\xibar(i)}(\alpha_i),(\sum_{j=0}^{m-1}p_1(\alpha_j))^{d-2}
(\sum_{l=0}^{m-1}p_2(\alpha_l))\rangle\\
&=&(d-2)!\sum_{l=0}^{m-1}\sum_{\sum_{j=0}^{m-1}n_j=d-2}\langle\prod_{i=0}^{m-1}S_{\xibar(i)}(\alpha_i),\prod_{j\neq l}
\frac{p_1(\alpha_j)^{n_j}}{n_j!}\frac{p_1(\alpha_l)^{n_l}}{n_l!}p_2(\alpha_l)\rangle\\
&=&(d-2)!\sum_{l=0}^{m-1}\langle\prod_{i=0}^{m-1}S_{\xibar(i)}(\alpha_i),\prod_{j\neq l}
\frac{p_1(\alpha_j)^{d_j}}{d_j!}\frac{p_1(\alpha_l)^{d_l-2}}{(d_l-2)!}p_2(\alpha_l)\rangle\\
&=&(d-2)!\sum_{l=0}^{m-1}\prod_{j\neq l}\langle S_{\xibar(j)}(\alpha_j),\frac{p_1(\alpha_j)^{d_j}}{d_j!}\rangle
\langle S_{\xibar(l)}(\alpha_l),\frac{p_1(\alpha_l)^{d_l-2}}{(d_l-2)!}p_2(\alpha_l)\rangle\\
&=&(d-2)!\sum_{l=0}^{m-1}\prod_{j\neq l}\frac{\textrm{dim}\xibar(j)}{d_j!}\frac{\kappa_{\xibar(l)}\textrm{dim}\xibar(l)}
{(d_l-2)!d_l(d_l-1)}\\
&=&(d-2)!\sum_{l=0}^{m-1}\kappa_{\xibar(l)}\prod_{j=0}^{m-1}\frac{\textrm{dim}\xibar(j)}{d_j!}\\
&=&\frac{1}{d(d-1)}\textrm{dim}\xibar\sum_{l=0}^{m-1}\kappa_{\xibar(l)}\\
\end{eqnarray*}
where in the sixth identity we used Proposition 3.3. Note that $|C(\mubar)|=\frac{|(\bZ_m)_d|}{Z_{\mubar}}$ for any $\bZ_m$-weighted partition $\mubar$. So we have
$$F_{\xibar}(\tau_0)=\frac{|C_{\tau_0}|X_{\xibar}(\tau_0)}{\textrm{dim}\xibar}=
\frac{d!m^dX_{\xibar}(\tau_0)}{2(d-2)!m^{d-1}\textrm{dim}\xibar}=\sum_{i=0}^{m-1}\frac{m\kappa_{\xibar(i)}}{2}$$
which proves the first identity. For the second identity, we have
\begin{eqnarray*}
X_{\xibar}(\rho_{\gamma_i})&=&\langle S_{\xibar},P_{\rho_{\gamma_i}}\rangle=\langle\prod_{a=0}^{m-1}S_{\xibar(a)}(\alpha_a),p_1(0)^{d-1}p_1(\gamma_i)\rangle\\
&=&\langle\prod_{a=0}^{m-1}S_{\xibar(a)}(\alpha_a),(\sum_{j=0}^{m-1}p_1(\alpha_j))^{d-1}
(\sum_{l=0}^{m-1}\alpha_l(-\gamma_i)p_1(\alpha_l))\rangle\\
&=&(d-1)!\sum_{l=0}^{m-1}\sum_{\sum_{j=0}^{m-1}n_j=d-1}\langle\prod_{a=0}^{m-1}S_{\xibar(a)}(\alpha_a),\prod_{j\neq l}
\frac{p_1(\alpha_j)^{n_j}}{n_j!}\frac{p_1(\alpha_l)^{n_l}}{n_l!}\alpha_l(-\gamma_i)p_1(\alpha_l)\rangle\\
&=&(d-1)!\sum_{l=0}^{m-1}\langle\prod_{a=0}^{m-1}S_{\xibar(a)}(\alpha_a),\prod_{j\neq l}
\frac{p_1(\alpha_j)^{d_j}}{d_j!}\frac{p_1(\alpha_l)^{d_l-1}}{(d_l-1)!}\alpha_l(-\gamma_i)p_1(\alpha_l)\rangle\\
&=&(d-1)!\sum_{l=0}^{m-1}\prod_{j\neq l}\langle S_{\xibar(j)}(\alpha_j),\frac{p_1(\alpha_j)^{d_j}}{d_j!}\rangle
\langle S_{\xibar(l)}(\alpha_l),\alpha_l(-\gamma_i)\frac{p_1(\alpha_l)^{d_l-1}}{(d_l-1)!}p_1(\alpha_l)\rangle\\
&=&(d-1)!\sum_{l=0}^{m-1}\prod_{j\neq l}\frac{\textrm{dim}\xibar(j)}{d_j!}\frac{\alpha_l(-\gamma_i)\textrm{dim}\xibar(l)}
{(d_l-1)!}\\
&=&(d-1)!\sum_{l=0}^{m-1}d_l\alpha_l(-\gamma_i)\prod_{j=0}^{m-1}\frac{\textrm{dim}\xibar(j)}{d_j!}\\
&=&\frac{1}{d}\textrm{dim}\xibar\sum_{l=0}^{m-1}d_l\alpha_l(-\gamma_i).
\end{eqnarray*}
Note that $|C(\rho_{\gamma_i})|=\frac{|(\bZ_m)_d|}{Z_{\rho_{\gamma_i}}}=\frac{d!m^d}{(d-1)!m^d}=d$, so we have
$$F_{\xibar}(\rho_{\gamma_i})=\frac{|C_{\rho_{\gamma_i}}|X_{\xibar}(\rho_{\gamma_i})}{\textrm{dim}\xibar}=
\sum_{j=0}^{m-1}|\xibar(j)|e^{\frac{-2\gamma_ij\pi\sqrt{-1}}{m}},$$
which proves the second identity.

\end{proof}

We summarize the above computation in the following theorem
\begin{theorem}[Burnside type formula for $H^\bullet_{\chi,\gamma}(\mubar,\nubar)_m$]
$$H^\bullet_{\chi,\gamma}(\mubar,\nubar)_m=\sum_{|\xibar|=d}
\frac{X_{\xibar}(\mubar)}{Z_{\mubar}}\frac{X_{\xibar}(\nubar)}{Z_{\nubar}}
F_{\xibar}(\tau_0)^r\prod_{i=1}^{l(\gamma)}F_{\xibar}(\rho_{\gamma_i}),$$
where $F_{\xibar}(\tau_0)=\sum_{i=0}^{m-1}\frac{m\kappa_{\xibar(i)}}{2}$ and
$F_{\xibar}(\rho_{\gamma_i})=\sum_{j=0}^{m-1}|\xibar(j)|
e^{\frac{-2\gamma_ij\pi\sqrt{-1}}{m}}$.
\end{theorem}

If we consider the generating function $\Phi^\bullet_{\mubar,\nubar}(\lambda;x)_m$ defined in section 3.1, it is easy to obtain the corresponding formula for $\Phi^\bullet_{\mubar,\nubar}(\lambda;x)_m$:
\begin{equation}\label{eqn:Burnside}
\Phi^\bullet_{\mubar,\nubar}(\lambda;x)_m=\sum_{|\xibar|=d}
\frac{X_{\xibar}(\mubar)}{Z_{\mubar}}\frac{X_{\xibar}(\nubar)}{Z_{\nubar}}e^{F_{\xibar}(\tau_0)\lambda}
\prod_{i=1}^{m-1}e^{F_{\xibar}(\rho_{i})x_i\lambda}
\end{equation}

\section{Virtual Localization}
In this section, we calculate $\xn{K}$ by virtual localization.

\subsection{Fixed points}
The connected components of the $\bC^*$ fixed points set of $\Mbar^\bullet_{\chi, \gamma}(\cY_s, \mubar)$ are parameterized by labeled graphs. We first introduce some graph notations which are similar to those in \cite{Liu-Liu-Zhou1}.

Let
$$\left[ f:(C,x_1,\ldots,x_{n}, y_1, \cdots , y_{l(\mubar)})\to \cY_s[a]\ \right]\in \Mbar^\bullet_{\chi, \gamma}(\cY_s, \mubar) $$
be a fixed point of the $\bC^*$-action. The restriction of the map
$$ \tilde{f}=\pi_s[a]\circ f: C\to \cY_s. $$
to an irreducible component of $C$ is either a constant map to one of the $\bC^*$
fixed points $p_0=0, p_1=\infty$ or a cover of $\cY_s$ which is fully
ramified over $p_0$ and $p_1$. We associate a labeled graph $\Gamma$ to
the $\bC^*$ fixed point
$$\left[ f:(C,x_1,\ldots,x_{n}, y_1, \cdots , y_{l(\mubar)})\to \cY_s[a]\ \right]\in \Mbar^\bullet_{\chi, \gamma}(\cY_s, \mubar)$$ as follows:
\begin{enumerate}
\item We assign a vertex $v$ to each connected
component $C_v$ of $\tilde{f}^{-1}(\{p_0,p_1\})$, a label
$i(v)=i$ if $\tilde{f}(C_v)=p_i$, where $i=0,1$, and a label $g(v)$
which is the arithmetic genus of $C_v$ (We define $g(v)=0$ if $C_v$ is a
point). We assign a set $n(v)$ of marked points on $C_v$. Denote by $V(\Gamma)^{(i)}$ the set of vertices with $i(v)=i$,
where $i=0,1$. Then the set $V(\Gamma)$ of vertices of the graph $\Gamma$
is a disjoint union of $V(\Gamma)^{(0)}$ and $V(\Gamma)^{(1)}$.

\item We assign an edge $e$ to each rational irreducible component $C_e$ of $C$
such that $\tilde{f}|_{C_e}$ is not a constant map.
Let $d(e)$ be the degree of $\tilde{f}|_{C_e}$ and $l(e)$ the monodromy around the unique point on $C_e$ which lies over $p_1$.
Then $\tilde{f}|_{C_e}$ is fully ramified over $p_0$ and $p_1$.
Let $E(\Gamma)$ denote the set of edges of $\Gamma$.

\item The set of flags of $\Gamma$ is given by
$$F(\Gamma)=\{(v,e):v\in V(\Gamma), e\in E(\Gamma),
C_v\cap C_e\neq \emptyset \}.$$

\item For each $v\in V(\Gamma)$, define
$$d(v)=\sum_{(v,e)\in F(\Gamma)}d(e),$$
and let $\nubar(v)$ be the $\bZ_m$-weighted partition of $d(v)$
determined by $\{(d(e),l(e)): (v,e)\in F(\Gamma)\}$ and let $\nubar$ be the $\bZ_m$-weighted partition of $d$ determined by $\{(d(e),l(e)): e\in E(\Gamma)\}$ .
When the target is $\cY_s[a]$, where $a>0$,
we assign an additional label for each
$v\in V(\Gamma)^{(1)}$:
let $\mubar(v)$ be the $\bZ_m$-weighted partition of
$d(v)$ determined by the ramification of
$f|_{C_v}:C_v\to\cY_s[a]$ over $p_1^{(a)}$.
\end{enumerate}
Note that for $v\in V(\Gamma)^{(1)}$,
$\nubar(v)$ has the same partition but the opposite monodromies with the $\bZ_m$-weighted partition of $d(v)$
determined by the ramification of $f|_{C_v}:C_v\to \cY_0(a)$
over $p_1^{(0)}$.

For any $e\in E(\Gamma)$, consider the map $f|_{C_e}:C_e\to \cY_s$. If the monodromy around the unique point on $C_e$ which lies over $p_1$ is $l(e)$, then by Lemma II.13 in \cite{Joh}, the monodromy around the unique point on $C_e$ which lies over $p_0$ is $-l(e)-d(e)s$. Let
$$w^s_{d(e)}(l(e))=-l(e)-d(e)s.$$
Similarly, for any $\bZ_m$-weighted partition $\nubar=\{(\nu_1,l_1),\cdots,(\nu_{l(\nu)},l_{l(\nu)})\}$, let
$$w^s(\nubar)=\{(\nu_1,w^s_{\nu_1}(l_1)),\cdots,(\nu_{l(\nu)},w^s_{\nu_{l(\nu)}}(l_{l(\nu)}))\}.$$

Let $\cM_{(\nu_i,l_i)}$ be the moduli space of $\bC^*$-fixed degree $\nu_i$ covers of $\cY_s$ with monodromies $l_i$ and $w^s_{\nu_i}(l_i)$ around $\infty$ and $0$ respectively. Let
\begin{eqnarray*}
J_{\chi,\mubar,\gamma}=\{(\chi^0,\chi^1,\nubar,\gamma^0,\gamma^1)|\chi^0,\chi^1\in 2\bZ,|\nubar|=|\mubar|,\gamma^0\sqcup  \gamma^1=\gamma,\\
-\chi^0+2l(\nubar)-\chi^1=-\chi,-\chi^0+2l(\nubar)\geq 0,-\chi^1+l(\nubar)+l(\mubar)\geq 0\}.
\end{eqnarray*}
Then the $\bC^*$-fixed locus can be identified with
\begin{eqnarray*}
\bigsqcup_{(\chi^0,\chi^1,\nubar,\gamma^0,\gamma^1)\in J_{\chi,\mubar,\gamma}}&&(\Mbar^\bullet_{\chi^0, \gamma^0-w^s_{\nu}(l)}(\cB\bZ_m)\times_{\bar{I}\cB\bZ_{m}^{l(\nubar)}}\cM_{(\nu_1,l_1)}
\times\cdots\times\cM_{(\nu_{l(\nubar)},l_{l(\nubar)})}\\
&&\times_{\bar{I}\cB\bZ_{m}^{l(\nubar)}}\Mbar^\bullet_{\chi^1, \gamma^1}(\cY_0,-\nubar, \mubar)//\bC^*)/\Aut(\nubar)
\end{eqnarray*}
where $\bar{I}\cB\bZ_{m}$ is the rigidified inertia stack of $\cB\bZ_{m}, l=(l_1,\cdots,l_{l(\nubar
)})$ and $w^s_{\nu}(l)=(w^s_{\nu_1}(l_1),\cdots,w^s_{\nu_{l(\nu)}}(l_{l(\nu)})$. Therefore, we can calculate our integral over
$$\bigsqcup_{(\chi^0,\chi^1,\nubar,\gamma^0,\gamma^1)\in J_{\chi,\mubar,\gamma}}\Mbar^\bullet_{\chi^0, \gamma^0-w^s_{\nu}(l)}(\cB\bZ_m)\times\Mbar^\bullet_{\chi^1, \gamma^1}(\cY_0,-\nubar, \mubar)//\bC^*$$
provided we include the following factor
$$\frac{1}{|\Aut(\nubar)|}\prod_{i=1}^{l(\nubar)}\frac{1}{m\nu_i}(\frac{m}{b_i})(\frac{m}{c_i})$$
where $b_i=\frac{m}{\textrm{gcd}(m,l_i)}$ and $c_i=\frac{m}{\textrm{gcd}(m,w^s_{\nu_i}(l_i))}$ are the orders of $l_i\in \bZ_m$ and $w^s_{\nu_i}(l_i)\in\bZ_m$ respectively.

We will use the following convention for the unstable integrals to simplify our expression
\begin{eqnarray*}
\int_{\Mbar_{0,(0)}(\cB\bZ_m)}\frac{1}{1-d\bar{\psi}}&=&\frac{1}{md^2}\\
\int_{\Mbar_{0,(c,-c)}(\cB\bZ_m)}\frac{1}{(1-d_1\bar{\psi}_1)(1
-d_2\bar{\psi_2})}
&=& \frac{1}{m(d_1+d_2)}\\
\int_{\Mbar_{0,(c,-c)}(\cB\bZ_m)}\frac{1}{(1-d\bar{\psi}_1)}=\frac{1}{md}
\end{eqnarray*}

\subsection{Contribution from each graph}
Calculations similar to those in section 4 of \cite{Zong} show that
$$\frac{i_{\Gamma}^*e_{\bC^*}(V^0)}{e_{\bC^*}(N_{\Gamma}^{\mathrm{vir} })}=A^{0}_0A^{1}_0$$
where
\begin{eqnarray*}
A^0_0&=&\sqrt{-1}^d\prod_{i=1}^{l(\nubar)}\nu_i\prod_{i=1}^{l(\nubar)}\frac{\prod_{j=0}^{\nu_i-1}(\nu_{i}\tau+\frac{l_i}{m}+j)}
{\nu_i!(u-\nu_{i}\bar{\psi_i})\frac{\Gcd(l_{i},m)}{m}}\left(\frac{\nu_{i}\tau+\frac{l_i}{m}}{\nu_i}u\right)^
{-\delta_{0,l_i}}\\
&&\cdot\prod_{v \in V(\Gamma)^{(0)}}(\sqrt{-1})^{|\nubar(v)|+l(\nubar(v))-2\sum_{i\in A''(\nubar(v))}\frac{m-l_i}{m}}
\Lambda_{g(v)}^{\vee,U}(\tau u)\Lambda_{g(v)}^{\vee,U^\vee}((-\tau-1)u)
\Lambda_{g(v)}^{\vee,1}(u)\\
&&\cdot(-1)^{-\delta_v}\left(\tau u(\tau+1)u\right)^{\sum_{(v,e)\in F(\Gamma)}\delta_{0,l(e)}-\delta_v}( u)^{l(\nubar(v))-1}\\
A^1_0&=& \left\{\begin{array}{ll}\sqrt{-1}^{l''(\mubar)-l'(\mubar)-2\sum_{i\in A''(\mubar)}\frac{k_i}{m}}, &\textrm{the target is }\cY_0\\
\sqrt{-1}^{l''(\mubar)-l'(\mubar)-2\sum_{i\in A''(\mubar)}\frac{k_i}{m}}\sqrt{-1}^{l(\gamma^1)-2\sum_{\gamma_i\in \gamma^1}\frac{\gamma_i}{m}}&\\
\cdot\prod_{i=1}^{l(\nubar)}\frac{m\nu_i}{\textrm{gcd}(m,l_i)}\frac{(\sqrt{-1}\tau u)^{-\chi^{1}+l(\mubar)+l(\nubar)+l(\gamma^1)}}{-u-\psi^0},  &\textrm{the target is }\cY_0[a],a>0
\end{array}\right.
\end{eqnarray*}
where
$$\delta_{v}=\left\{\begin{array}{ll}1, &\textrm{if all monodromies around loops on $C_v$ are trivial}\\
0, &\textrm{otherwise}.\end{array} \right.$$

For $1\leq s\leq m-1$, we have
$$\frac{i_{\Gamma}^*e_{\bC^*}(V^s)}{e_{\bC^*}(N_{\Gamma}^{\mathrm{vir} })}=A^{0}_sA^{1}_s$$
where
\begin{eqnarray*}
A^0_s&=&\sqrt{-1}^d\prod_{i=1}^{l(\nubar)}\nu_i\prod_{i=1}^{l(\nubar)}
\frac{\prod_{j=1}^{\nu_i}(-\frac{w^s_{\nu_i}(l_i)}{m}+j)u^{-\delta_{0,w^s_{\nu_i}(l_i)}}}
{\nu_i!(u-\nu_{i}\bar{\psi_i})\frac{\Gcd(l_{i},m)}{m}}\\
&&\cdot\prod_{v \in V(\Gamma)^{(0)}}(\sqrt{-1})^{|\nubar(v)|+l(\nubar(v))-2(\frac{|\nubar(v)|s}{m}+\sum_{i=1}^{l(\nubar(v))}\frac{w^s_{\nu_i}(l_i)}{m}})
\Lambda_{g(v)}^{\vee,U}(0)\Lambda_{g(v)}^{\vee,U^\vee}(-u)
\Lambda_{g(v)}^{\vee,1}(u)\\
&&\cdot(-1)^{\delta_v}(0)^{\sum_{(v,e)\in F(\Gamma)}\delta_{0,w^s_{d(e)}(l(e))}-\delta_v}(u)^{2(\sum_{(v,e)\in F(\Gamma)}\delta_{0,w^s_{d(e)}(l(e))}-\delta_v)+l(\nubar(v))-1}\\
A^1_s&=& \left\{\begin{array}{ll}\sqrt{-1}^{l''(\mubar)-l'(\mubar)-2\sum_{i\in A''(\mubar)}\frac{k_i}{m}}(\frac{su}{m})^{-l'(\mubar)}, &\textrm{the target is }\cY_s\\
\sqrt{-1}^{l''(\mubar)-l'(\mubar)-2\sum_{i\in A''(\mubar)}\frac{k_i}{m}}\sqrt{-1}^{l(\gamma^1)-2\sum_{\gamma_i\in \gamma^1}\frac{\gamma_i}{m}}&\\
\cdot(\frac{su}{m})^{-l'(\mubar)}\prod_{i=1}^{l(\nubar)}\frac{m\nu_i}{\textrm{gcd}(m,l_i)}
\frac{(\frac{\sqrt{-1}su}{m})^{-\chi^{1}+l(\mubar)+l(\nubar)+l(\gamma^1)}}{-u-\psi^0},  &\textrm{the target is }\cY_s[a],a>0
\end{array}\right.
\end{eqnarray*}

\subsection{Proof of Theorem 1}
\begin{eqnarray*}
&&K^{\bullet 0}_{\chi,\mubar,\gamma}\\
&=&\frac{1}{\am}\int_{[\Mbar^\bullet_{\chi, \gamma}(\cY_0, \mubar)]^{\vir}}e(V^0)\\
&=&\frac{1}{\am}\sum_{(\chi^0,\chi^1,\nubar,\gamma^0,\gamma^1)\in J_{\chi,\mubar,\gamma}}\frac{1}{|\Aut(\nubar)|}\prod_{i=1}^{l(\nubar)}\frac{1}{m\nu_i}(\frac{m}{b_i})^2\\
&&\cdot\int_{[\Mbar^\bullet_{\chi^0, \gamma^0+l}(\cB\bZ_m)\times\Mbar^\bullet_{\chi^1, \gamma^1}(\cY_0,-\nubar, \mubar)//\bC^*]^{\vir}}\frac{i_{\Gamma}^*e_{\bC^*}(V^0)}{e_{\bC^*}(N_{\Gamma}^{\mathrm{vir} })}\\
&=&\sqrt{-1}^{d+l''(\mubar)-l'(\mubar)-2\sum_{i\in A''(\mubar)}\frac{k_i}{m}}\sum_{(\chi^0,\chi^1,\nubar,\gamma^0,\gamma^1)\in J_{\chi,\mubar,\gamma}}\\
&&\left(G^\bullet_{\chi^0,\nubar,\gamma^0}(\tau)_{m}\cdot Z_{\nubar}\frac{(-\sqrt{-1}\tau)^{-\chi^1 +l(\nubar) +l(\mubar)+l(\gamma^1)}}
{(-\chi^1 +l(\nubar) +l(\mubar))!}H^\bullet_{\chi^1,\gamma^1}(\mubar,-\nubar)_m\cdot\sqrt{-1}^{l(\gamma^1)-2\sum_{\gamma_i\in \gamma^1}\frac{\gamma_i}{m}}\right)
\end{eqnarray*}
Define the generating function $K^{\bullet 0}_{\mubar}(\lambda;x)$ to be
$$K^{\bullet 0}_{\mubar }(\lambda;x)=\sqrt{-1}^{-(d+l''(\mubar)-l'(\mubar)-2\sum_{i\in A''(\mubar)}\frac{k_i}{m})}\sum_{\chi,\gamma}\lambda^{-\chi+l(\mubar)+l(\gamma)}K^{\bullet 0}_{\chi,\mubar,\gamma}\frac{x_\gamma}{\gamma!}.$$
Then we have
$$K^{\bullet 0}_{\mubar }(\lambda;x)=\sum_{|\nubar|=|\mubar|}G^\bullet_{\nubar}(\lambda;\tau;x)_mZ_{\nubar}
\Phi^\bullet_{-\nubar,\mubar}(-\sqrt{-1}\tau\lambda;\sqrt{-1}^{1-\frac{2}{m}}x_1,\cdots,\sqrt{-1}^{1-\frac{2i}{m}}x_i,
\cdots,\sqrt{-1}^{1-\frac{2(m-1)}{m}}x_{m-1})_m.$$
Let $\tau=0$ we have
$$K^{\bullet 0}_{\mubar}(\lambda;x)=G^\bullet_{\mubar}(\lambda;0;x).$$
Define $G_d(\tau)=(G^\bullet_{\nubar}(\lambda;\tau;x)_m)_{|\nubar|=d}$ and $G_d(0)=(G^\bullet_{\mubar}(\lambda;0;x)_m)_{|\mubar|=d}$ to be two column vectors indexed by $\nubar$ and $\mubar$ respectively. Let $\Phi_d(\tau)=(\Phi_d^{\mubar,\nubar}(\tau))_{|\mubar|=d,|\nubar|=d}$ be a matrix indexed by $\nubar$ and $\mubar$, where
$$\Phi_d^{\mubar,\nubar}(\tau)=Z_{\nubar}
\Phi^\bullet_{-\nubar,\mubar}(-\sqrt{-1}\tau\lambda;\sqrt{-1}^{1-\frac{2}{m}}x_1,\cdots,\sqrt{-1}^{1-\frac{2i}{m}}x_i,
\cdots,\sqrt{-1}^{1-\frac{2(m-1)}{m}}x_{m-1})_m.$$
$\Phi_d(\tau)$ is invertible because if we view its entries as elements in $\bC[[\lambda,x]]$ then only the diagonal entries have constant terms. So we have
\begin{equation}\label{eqn:framing1}
G_d(\tau)=\Phi_d(\tau)^{-1}G_d(0)
\end{equation}
By the orthogonality of characters and (\ref{eqn:Burnside}), it is easy to see that
\begin{equation}\label{eqn:othogonality1}
\Phi^\bullet_{\mubar,\nubar}(\lambda_1+\lambda_2,x)_m=\sum_{\xibar}\Phi^\bullet_{\mubar,\xibar}(\lambda_1,x)_m Z_{\xibar}
\Phi^\bullet_{-\xibar,\nubar}(\lambda_2,x)_m,
\end{equation}
and
\begin{equation}\label{eqn:othogonality2}
\Phi^\bullet_{\mubar,\nubar}(0,x)_m=\frac{1}{Z_{\mubar}}\delta_{\mubar,-\nubar}.
\end{equation}
Therefore, (\ref{eqn:framing1}) is equivalent to
\begin{equation}\label{eqn:framing2}
G^\bullet_{\mubar}(\lambda;\tau;x)_m=\sum_{|\nubar|=|\mubar|}G^\bullet_{\nubar}(\lambda;0;x)_m Z_{\nubar}
\Phi^\bullet_{-\nubar,\mubar}(\sqrt{-1}\tau\lambda;\tilde{x})_m,
\end{equation}
where $\tilde{x}=(\sqrt{-1}^{1-\frac{2}{m}}x_1,\cdots,\sqrt{-1}^{1-\frac{2i}{m}}x_i,
\cdots,\sqrt{-1}^{1-\frac{2(m-1)}{m}}x_{m-1})$. This finishes the proof of Theorem 1.

\subsection{Proof of Theorem 2}
For any $1\leq s\leq m-1$ and $l'(\mubar)\neq 0$, we have
\begin{eqnarray*}
0&=&\xn{K}\\
&=&\frac{1}{\am}\int_{[\Mbar^\bullet_{\chi, \gamma}(\cY_s, \mubar)]^{\vir}}e(V^s)\\
&=&\frac{1}{\am}\sum_{(\chi^0,\chi^1,\nubar,\gamma^0,\gamma^1)\in J_{\chi,\mubar,\gamma}}\frac{1}{|\Aut(\nubar)|}\prod_{i=1}^{l(\nubar)}\frac{1}{m\nu_i}(\frac{m}{b_i})(\frac{m}{c_i})\\
&&\cdot\int_{[\Mbar^\bullet_{\chi^0, \gamma^0-w^s_{\nu}(l)}(\cB\bZ_m)\times\Mbar^\bullet_{\chi^1, \gamma^1}(\cY_s,-\nubar, \mubar)//\bC^*]^{\vir}}\frac{i_{\Gamma}^*e_{\bC^*}(V^s)}{e_{\bC^*}(N_{\Gamma}^{\mathrm{vir} })}\\
&=&\sqrt{-1}^{d+l''(\mubar)-l'(\mubar)-2(\frac{s}{m}|\mubar|+\sum_{i\in A''(\mubar)}\frac{k_i}{m})}(\frac{s}{m})^{-l'(\mubar)}\sum_{(\chi^0,\chi^1,\nubar,\gamma^0,\gamma^1)\in J_{\chi,\mubar,\gamma}}\\
&&\left(G^\bullet_{\chi^0,-w^s(\nubar),\gamma^0}(0)_{m}\cdot Z_{\nubar}\frac{(-\frac{\sqrt{-1}s}{m})^{-\chi^1 +l(\nubar) +l(\mubar)+l(\gamma^1)}}
{(-\chi^1 +l(\nubar) +l(\mubar))!}H^\bullet_{\chi^1,\gamma^1}(\mubar,-\nubar)_m\cdot\sqrt{-1}^{l(\gamma^1)-2\sum_{\gamma_i\in \gamma^1}\frac{\gamma_i}{m}}\right)
\end{eqnarray*}
Define the generating function $K^{\bullet s}_{\mubar}(\lambda;x)$ to be
$$K^{\bullet s}_{\mubar }(\lambda;x)=\sqrt{-1}^{-(d+l''(\mubar)-l'(\mubar)-2(\frac{s}{m}|\mubar|+\sum_{i\in A''(\mubar)}\frac{k_i}{m}))}(\frac{s}{m})^{l'(\mubar)}\sum_{\chi,\gamma}\lambda^{-\chi+l(\mubar)+l(\gamma)}K^{\bullet s}_{\chi,\mubar,\gamma}\frac{x_\gamma}{\gamma!}.$$
Then we have
\begin{equation}\label{eqn:key}
0=K^{\bullet s}_{\mubar }(\lambda;x)
=\sum_{|\nubar|=|\mubar|}G^\bullet_{-w^s(\nubar)}(\lambda;0;x)_mZ_{\nubar}
\Phi^\bullet_{-\nubar,\mubar}(-\frac{\sqrt{-1}s}{m}\lambda;\tilde{x})_m,
\end{equation}
where $\tilde{x}=(\sqrt{-1}^{1-\frac{2}{m}}x_1,\cdots,\sqrt{-1}^{1-\frac{2i}{m}}x_i,
\cdots,\sqrt{-1}^{1-\frac{2(m-1)}{m}}x_{m-1})$.

Now notice that when $k_i=0$ for some $i\in\{1,\cdots,l(\mubar)\}$, if $G_{g,\mubar,\gamma}(0)_m$ is nonzero, then the following two conditions must be satisfied
\begin{enumerate}
\item $l(\mubar)=1$.

\item $\gamma=\emptyset$.
\end{enumerate}
Let $\cM$ be the connected component of $\Mbar_{g,(0)}(\cB\bZ_m)$ such that the monodromies around the $2g$ noncontractible loops are trivial. Then
$$G_{g,\{(d,0)\},\emptyset}(0)_m=-\sqrt{-1}^{d+1}\int_{\cM}\frac{\Lambda_{g}^{\vee,U}(0 )\Lambda_{g}^{\vee,U^\vee}(-1)
\Lambda_{g}^{\vee,1}(1)}{1-d\bar{\psi}_1}$$
There is a canonical map $\rho:\cM \to \Mbar_{g,1}$ with deg$\rho=\frac{1}{m}$. On the other hand, we have
$$E^U|_{\cM}\cong E^{U^\vee}|_{\cM}\cong E^1|_{\cM}.$$
Therefore
\begin{eqnarray*}
G_{g,\{(d,0)\},\emptyset}(0)_m&=&-\frac{\sqrt{-1}^{d+1}}{m}\int_{\Mbar_{g,1}}\frac{\Lambda_{g}^{\vee}(0 )\Lambda_{g}^{\vee}(-1)
\Lambda_{g}^{\vee}(1)}{1-d\psi_1}\\
&=&-\frac{\sqrt{-1}^{d+1}}{m}d^{2g-2}\int_{\Mbar_{g,1}}\lambda_g\psi_1^{2g-2}
\end{eqnarray*}
So we have
$$G_{\{(d,0)\}}(\lambda;0;x)_m=-\frac{\sqrt{-1}^{d+1}}{2md\sin(\frac{d\lambda}{2})}.$$
Therefore, for any $\mubar$ with $k_1=\cdots=k_{l(\mubar)}=0$ we have
\begin{equation}\label{eqn:initial}
G^\bullet_{\mubar}(\lambda;0;x)_m=\frac{1}{|\Aut(\mubar)|}
\prod_{i=1}^{l(\mubar)}\left(-\frac{\sqrt{-1}^{\mu_i+1}}{2m\mu_i\sin(\frac{\mu_i\lambda}{2})}\right).
\end{equation}

Now for any $-w^s(\nubar)$, let $\xibar=\{(\nu_i,-w^s_{\nu_i}(l_i))|w^s_{\nu_i}(l_i)=0\}$ and $\etabar=\{(\nu_i,-w^s_{\nu_i}(l_i))|w^s_{\nu_i}(l_i)\neq0\}$. Then $-w^s(\nubar)=\xibar\sqcup\etabar$ and
\begin{equation}\label{eqn:prod}
G^\bullet_{-w^s(\nubar)}(\lambda;0;x)_m=G^\bullet_{\xibar}(\lambda;0;x)_mG^\bullet_{\etabar}(\lambda;0;x)_m
\end{equation}
because of condition (1) above. Let $\etabar=\{(\eta_1,h_1),\cdots,(\eta_{l(\eta)},h_{l(\eta)})\}$ with $h_1,\cdots,h_{l(\eta)}$ nontrivial. Let $c=\Gcd(m,\eta_1)$ and let $\bar{h}_1\in \{0,\cdots,c-1\}$ denote $h_1$(mod $c$). Let
$$\Sigma_{\etabar}=\{s\in\{1,\cdots,m-1\}|-h_1+\eta_1s=-\bar{h}_1\in \bZ_m\}$$
Then we have
$$|\Sigma_{\etabar}|=\left\{\begin{array}{ll}c-1,&\textrm{if }h_1\in\{1,\cdots,c-1\}\\
c,&\textrm{otherwise}\end{array} \right.$$
If we view $\Sigma_{\etabar}$ as a subset of $\{1,\cdots,m-1\}$, we can give $\Sigma_{\etabar}$ an order: $\Sigma_{\etabar}=\{s_1,\cdots,s_{|\Sigma_{\etabar}|}\}, s_i< s_j$ if $i<j$. Define $s(\etabar)\in\Sigma_{\etabar}$ to be
$$s(\etabar)=\left\{\begin{array}{ll}s_{\bar{h}_1},&\textrm{if }h_1\in\{1,\cdots,c-1\}\\
s_{\bar{h}_1+1},&\textrm{otherwise}\end{array} \right.$$
Let
\begin{eqnarray*}
B_d&=&\{\etabar||\etabar|\leq d,l(\etabar)=l''(\etabar)\}\\
C_d&=&\{(\mubar,s)|\mu=\eta,s=s(\etabar),k_1=0,-w^s(\mubar\setminus \{(\mu_1,k_1)\})=\etabar\setminus \{(\eta_1,h_1)\},|\etabar|\leq d,l(\etabar)=l''(\etabar)\}.
\end{eqnarray*}
Let $\tilde{x}=(\sqrt{-1}^{1-\frac{2}{m}}x_1,\cdots,\sqrt{-1}^{1-\frac{2i}{m}}x_i,
\cdots,\sqrt{-1}^{1-\frac{2(m-1)}{m}}x_{m-1})$, then we define
$$\beta_d=(-\sum_{|\xibar|=|\mubar|,l(\xibar)=l'(\xibar)}G^\bullet_{\xibar}(\lambda;0;x)_mZ_{w^s(\xibar)}
\tilde{\Phi}^\bullet_{-w^s(\xibar),\mubar}(-\frac{\sqrt{-1}s}{m}\lambda;\tilde{x})_m)_{(\mubar,s)\in C_d}$$
and
$$G'_d=(G^\bullet_{\etabar}(\lambda;0;x)_m)_{\etabar\in B_d}$$
to be two column vectors indexed by $(\mubar,s)$ and $\etabar$ respectively. Let $$\tilde{\Phi}_d(\lambda;x)=(\tilde{\Phi}_d^{(\mubar,s),\etabar}(\lambda;x))_{(\mubar,s)\in C_d,\etabar\in B_d}$$
be a matrix indexed by $(\mubar,s)$ and $\etabar$, where
$$\tilde{\Phi}_d^{(\mubar,s),\etabar}(\lambda;x)=\left\{\begin{array}{ll}0,&\textrm{if}|\etabar|>|\mubar|\\
Z_{w^s(-\etabar)}
\tilde{\Phi}^\bullet_{-w^s(-\etabar),\mubar}(-\frac{\sqrt{-1}s}{m}\lambda;\tilde{x})_m,&\textrm{if}|\etabar|=|\mubar|\\
\sum_{|\xibar|=|\mubar|-|\etabar|,l(\xibar)=l'(\xibar)}G^\bullet_{\xibar}(\lambda;0;x)_mZ_{w^s(-(\xibar\sqcup\etabar))}
\tilde{\Phi}^\bullet_{-w^s(-(\xibar\sqcup\etabar)),\mubar}(-\frac{\sqrt{-1}s}{m}\lambda;\tilde{x})_m
&\textrm{if}|\etabar|<|\mubar|\end{array} \right.$$
We will show that $\tilde{\Phi}_d(\lambda;x)$ is invertible over $\bC((\lambda,x))$ in Appendix A.
Then by (\ref{eqn:key}) we have
$$G'_d=\tilde{\Phi}_d^{-1}(\lambda;x)\beta_d$$
This finishes the proof of Theorem 2.

\section{Examples}
In this section, we will compute the degree 1 and degree 2 $\bZ_2$ Gromov-Witten vertices. Then we will use these results to compute some $\bZ_2$ -Hodge integrals which appear in \cite{Ros}.

\subsection{degree 1 case}
The degree 1 $\bZ_2$ Gromov-Witten vertex has been computed in \cite{Ros}. We use our formula to recompute it and use this result to compute the degree 2 $\bZ_2$ Gromov-Witten vertices in the next subsection.

When $d=1$ and $m=2$, let $\chi=2-2g$ and $n=l(\gamma)$, then we have
\begin{eqnarray*}
H^\bullet_{\chi,\gamma}(\mubar,\nubar)_2&=&H^\circ_{g,\gamma}(\mubar,\nubar)_2\\
&=&\delta_{0,\langle \frac{n+k_1+l_1}{2}\rangle}2^{2g-1}H^\circ_{g,n}(\{(1)\},\{(1)\})\\
&=&\delta_{0,\langle \frac{n+k_1+l_1}{2}\rangle}2^{2g-1}H^\circ_{g}(\{(1)\},\{(1)\})\\
&=&\delta_{0,g}\delta_{0,\langle \frac{n+k_1+l_1}{2}\rangle}2^{2g-1}
\end{eqnarray*}
In particular
\begin{eqnarray*}
H^\bullet_{\chi,\gamma}(\{(1,0)\},\{(1,0)\})&=&\delta_{0,g}\delta_{0,\langle \frac{n}{2}\rangle}2^{2g-1}\\
H^\bullet_{\chi,\gamma}(\{(1,1)\},\{(1,0)\})&=&\delta_{0,g}\delta_{0,\langle \frac{n+1}{2}\rangle}2^{2g-1}
\end{eqnarray*}
So we have
\begin{eqnarray*}
\Phi^\bullet_{\{(1,0)\},\{(1,0)\}}(-\frac{\sqrt{-1}}{2}\lambda;x)_2&=&\frac{1}{2}\cos(\frac{\lambda x}{2})\\
\Phi^\bullet_{\{(1,1)\},\{(1,0)\}}(-\frac{\sqrt{-1}}{2}\lambda;x)_2&=&-\frac{\sqrt{-1}}{2}\sin(\frac{\lambda x}{2})
\end{eqnarray*}
By theorem 2
$$2\Phi^\bullet_{\{(1,0)\},\{(1,0)\}}(-\frac{\sqrt{-1}}{2}\lambda;x)_2G^\bullet_{\{(1,1)\}}(\lambda;0;x)_2
=-2\Phi^\bullet_{\{(1,1)\},\{(1,0)\}}(-\frac{\sqrt{-1}}{2}\lambda;x)_2G^\bullet_{\{(1,0)\}}(\lambda;0;x)_2$$
By equation (\ref{eqn:initial})
$$G^\bullet_{\{(1,0)\}}(\lambda;0;x)_2=\frac{1}{4\sin(\frac{\lambda}{2})}$$
Therefore, we obtain
\begin{equation}\label{eqn:deg1}
G^\bullet_{\{(1,1)\}}(\lambda;0;x)_2=\frac{\sqrt{-1}}{4\sin(\frac{\lambda}{2})}\tan (\frac{\lambda x}{2})
\end{equation}
By theorem 1
\begin{eqnarray*}
G^\bullet_{\{(1,1)\}}(\lambda;0;x)_2&=&2\Phi^\bullet_{\{(1,1)\},\{(1,1)\}}(-\tau\sqrt{-1}\lambda;x)_2
G^\bullet_{\{(1,1)\}}(\lambda;\tau;x)_2\\
&&+2\Phi^\bullet_{\{(1,1)\},\{(1,0)\}}(-\tau\sqrt{-1}\lambda;x)_2
G^\bullet_{\{(1,0)\}}(\lambda;\tau;x)_2\\
G^\bullet_{\{(1,0)\}}(\lambda;0;x)_2&=&2\Phi^\bullet_{\{(1,1)\},\{(1,0)\}}(-\tau\sqrt{-1}\lambda;x)_2
G^\bullet_{\{(1,1)\}}(\lambda;\tau;x)_2\\
&&+2\Phi^\bullet_{\{(1,0)\},\{(1,0)\}}(-\tau\sqrt{-1}\lambda;x)_2
G^\bullet_{\{(1,0)\}}(\lambda;\tau;x)_2
\end{eqnarray*}
Similar to the computation above, we have $\Phi^\bullet_{\{(1,1)\},\{(1,1)\}}(-\tau\sqrt{-1}\lambda;x)_2=\frac{1}{2}\cos (\tau\lambda x)$. So the degree 1 framed vertices are given by
\begin{eqnarray*}
G^\bullet_{\{(1,1)\}}(\lambda;\tau;x)_2&=&\frac{\sqrt{-1}\sin(\tau\lambda x)}{4\sin(\frac{\lambda}{2})}+\frac{\sqrt{-1}\cos (\tau\lambda x)}{4\sin(\frac{\lambda}{2})}\tan (\frac{\lambda x}{2})\\
G^\bullet_{\{(1,0)\}}(\lambda;\tau;x)_2&=&-\frac{\sin(\tau\lambda x)}{4\sin(\frac{\lambda}{2})}\tan (\frac{\lambda x}{2})
+\frac{\cos (\tau\lambda x)}{4\sin(\frac{\lambda}{2})}
\end{eqnarray*}

\subsection{degree 2 case}
In this section, we compute the degree 2 $\bZ_2$ Gromov-Witten vertices for $\tau=0$. Then we use these results to compute the predictions of the $\bZ_2$-Hodge integrals in \cite{Ros}. These results can be viewed as an evidence for the conjecture of the orbifold GW/DT correspondence in \cite{Ros}.

When $d=2$, we only need to notice that all degree 2 double Hurwitz numbers are $\frac{1}{2}$. Then calculations similar to those in degree 1 case show that
\begin{eqnarray*}
\Phi^\bullet_{\{(2,0)\},\{(2,0)\}}(-\frac{\sqrt{-1}}{2}\lambda;x)_2&=&\frac{1}{4}\cos(\lambda)
\cos(\lambda x) \\
\Phi^\bullet_{\{(2,0)\},\{(1,1),(1,1)\}}(-\frac{\sqrt{-1}}{2}\lambda;x)_2&=&-\frac{\sqrt{-1}}{8}\sin(\lambda)
\cos(\lambda x) \\
\Phi^\bullet_{\{(2,0)\},\{(2,1)\}}(-\frac{\sqrt{-1}}{2}\lambda;x)_2&=&-\frac{\sqrt{-1}}{4}\cos(\lambda)
\sin(\lambda x) \\
\Phi^\bullet_{\{(1,1),(1,0)\},\{(2,0)\}}(-\frac{\sqrt{-1}}{2}\lambda;x)_2&=&-\frac{1}{4}\sin(\lambda)
\sin(\lambda x)\\
\Phi^\bullet_{\{(1,0),(1,0)\},\{(2,0)\}}(-\frac{\sqrt{-1}}{2}\lambda;x)_2&=&-\frac{\sqrt{-1}}{8}\sin(\lambda)
\cos(\lambda x)\\
\Phi^\bullet_{\{(1,0),(1,0)\},\{(2,1)\}}(-\frac{\sqrt{-1}}{2}\lambda;x)_2&=&-\frac{1}{8}\sin(\lambda)
\sin(\lambda x)\\
\Phi^\bullet_{\{(1,0),(1,0)\},\{(1,0),(1,0)\}}(-\frac{\sqrt{-1}}{2}\lambda;x)_2&=&\frac{1}{16}(\cos(\lambda)
\cos(\lambda x)+1)\\
\Phi^\bullet_{\{(1,1),(1,0)\},\{(1,0),(1,0)\}}(-\frac{\sqrt{-1}}{2}\lambda;x)_2&=&-\frac{\sqrt{-1}}{8}\cos(\lambda)
\sin(\lambda x)\\
\Phi^\bullet_{\{(1,0),(1,0)\},\{(1,1),(1,1)\}}(-\frac{\sqrt{-1}}{2}\lambda;x)_2&=&\frac{1}{16}(\cos(\lambda)
\cos(\lambda x)-1)\\
\end{eqnarray*}
By equation (\ref{eqn:key}), we have
\begin{eqnarray*}
0&=&4G^\bullet_{\{(2,1)\}}(\lambda;0;x)_2\Phi^\bullet_{\{(2,1)\},\{(2,0)\}}(-\frac{\sqrt{-1}}{2}\lambda;x)_2
+4G^\bullet_{\{(2,0)\}}(\lambda;0;x)_2\Phi^\bullet_{\{(2,0)\},\{(2,0)\}}(-\frac{\sqrt{-1}}{2}\lambda;x)_2\\
&&+8G^\bullet_{\{(1,1),(1,1)\}}(\lambda;0;x)_2
\Phi^\bullet_{\{(1,0),(1,0)\},\{(2,0)\}}(-\frac{\sqrt{-1}}{2}\lambda;x)_2\\
&&+4G^\bullet_{\{(1,1),(1,0)\}}(\lambda;0;x)_2
\Phi^\bullet_{\{(1,1),(1,0)\},\{(2,0)\}}(-\frac{\sqrt{-1}}{2}\lambda;x)_2\\
&&+8G^\bullet_{\{(1,0),(1,0)\}}(\lambda;0;x)_2
\Phi^\bullet_{\{(1,1),(1,1)\},\{(2,0)\}}(-\frac{\sqrt{-1}}{2}\lambda;x)_2
\end{eqnarray*}
\begin{eqnarray*}
0&=&4G^\bullet_{\{(2,1)\}}(\lambda;0;x)_2\Phi^\bullet_{\{(2,1)\},\{(1,0),(1,0)\}}(-\frac{\sqrt{-1}}{2}\lambda;x)_2\\
&&+4G^\bullet_{\{(2,0)\}}(\lambda;0;x)_2\Phi^\bullet_{\{(2,0)\},\{(1,0),(1,0)\}}(-\frac{\sqrt{-1}}{2}\lambda;x)_2\\
&&+8G^\bullet_{\{(1,1),(1,1)\}}(\lambda;0;x)_2
\Phi^\bullet_{\{(1,0),(1,0)\},\{(1,0),(1,0)\}}(-\frac{\sqrt{-1}}{2}\lambda;x)_2\\
&&+4G^\bullet_{\{(1,1),(1,0)\}}(\lambda;0;x)_2
\Phi^\bullet_{\{(1,1),(1,0)\},\{(1,0),(1,0)\}}(-\frac{\sqrt{-1}}{2}\lambda;x)_2\\
&&+8G^\bullet_{\{(1,0),(1,0)\}}(\lambda;0;x)_2
\Phi^\bullet_{\{(1,1),(1,1)\},\{(1,0),(1,0)\}}(-\frac{\sqrt{-1}}{2}\lambda;x)_2\\
\end{eqnarray*}
By equation (\ref{eqn:initial}) (\ref{eqn:prod}) (\ref{eqn:deg1})we have
\begin{eqnarray*}
G^\bullet_{\{(2,0)\}}(\lambda;0;x)_2&=&\frac{\sqrt{-1}}{8\sin(\lambda)}\\
G^\bullet_{\{(1,0),(1,0)\}}(\lambda;0;x)_2&=&\frac{1}{32\sin^2(\frac{\lambda}{2})}\\
G^\bullet_{\{(1,1),(1,0)\}}(\lambda;0;x)_2&=&\frac{\sqrt{-1}}{16\sin^2(\frac{\lambda}{2})}\tan(\frac{\lambda x}{2})
\end{eqnarray*}
Therefore the two nontrivial degree 2 vertices are given by
\begin{eqnarray*}
G^\bullet_{\{(2,1)\}}(\lambda;0;x)_2&=&-\frac{1}{8}\tan(\frac{\lambda x}{2})\frac{\cos(\lambda x)+\cos(\lambda)+1}
{\sin(\lambda)(\cos(\lambda x)+\cos(\lambda))}\\
G^\bullet_{\{(1,1),(1,1)\}}(\lambda;0;x)_2&=&-\frac{1}{16\cos^2(\frac{\lambda x}{2})(\cos(\lambda x)+\cos(\lambda))}
-\frac{1}{32\sin^2(\frac{\lambda}{2})}\tan^2(\frac{\lambda x}{2})
\end{eqnarray*}

Now we use these two vertices to compute the predictions in \cite{Ros}. The following Hodge integrals are defined in \cite{Ros}:
\begin{eqnarray*}
G(1,g):&=&\sum_{\gamma}\int_{\Mbar_{g,\gamma}(\cB\bZ_2)}\frac{\Lambda_{g}^{\vee,U}(0 )\Lambda_{g}^{\vee,U^\vee}(-1)
\Lambda_{g}^{\vee,1}(1)}{\frac{1}{2}-\bar{\psi}_1}\frac{x^{l(\gamma)-1}}{(l(\gamma)-1)!}\\
G(2,g):&=&\sum_{\gamma}\int_{\Mbar_{g,\gamma}(\cB\bZ_2)}\frac{\Lambda_{g}^{\vee,U}(0 )\Lambda_{g}^{\vee,U^\vee}(-1)
\Lambda_{g}^{\vee,1}(1)}{(1-\bar{\psi}_1)(1-\bar{\psi}_2)}\frac{x^{l(\gamma)-2}}{(l(\gamma)-2)!}
\end{eqnarray*}
where $\gamma=(1,\cdots,1)$ is a vector of nontrivial elements in $\bZ_2$.

Notice that
\begin{eqnarray*}
G_{\{(2,1)\}}(\lambda;0;x)_2&=&G^\bullet_{\{(2,1)\}}(\lambda;0;x)_2\\
G_{\{(1,1),(1,1)\}}(\lambda;0;x)_2&=&
G^\bullet_{\{(1,1),(1,1)\}}(\lambda;0;x)_2-\frac{1}{2}G_{\{(1,1)\}}(\lambda;0;x)_2^2\\
&=&-\frac{1}{16\cos^2(\frac{\lambda x}{2})(\cos(\lambda x)+\cos(\lambda))}
\end{eqnarray*}
So if we define $a_1(x),a_2(x),a_3(x)$ to be the coefficients of $\lambda,\lambda^3,\lambda^5$ in $G_{\{(2,1)\}}(\lambda;0;\frac{x}{\lambda})_2$ respectively, then we have
\begin{eqnarray*}
G(1,1)&=&\frac{16}{3}a_1(x)=\int\frac{-1}{12}\sec^4(\frac{x}{2})+\frac{5}{24}\sec^6(\frac{x}{2})dx\\
G(1,2)&=&\frac{16}{3}a_2(x)=\int\frac{-1}{240}\sec^4(\frac{x}{2})-\frac{13}{288}\sec^6(\frac{x}{2})
+\frac{7}{96}\sec^8(\frac{x}{2})dx\\
G(1,3)&=&\frac{16}{3}a_3(x)=\int\frac{-11}{30240}\sec^4(\frac{x}{2})+\frac{1}{576}\sec^6(\frac{x}{2})
-\frac{1}{48}\sec^8(\frac{x}{2})+\frac{3}{128}\sec^{10}(\frac{x}{2})dx
\end{eqnarray*}
This gives the first three predictions in \cite{Ros}. Similarly, if we define $b_1(x),b_2(x),b_3(x)$ to be the coefficients of $\lambda^2,\lambda^4,\lambda^6$ in $G_{\{(1,1),(1,1)\}}(\lambda;0;\frac{x}{\lambda})_2$ respectively, then we have
\begin{eqnarray*}
G(2,1)&=&-8b_1(x)=\frac{1}{16}\sec^6(\frac{x}{2})\\
G(2,2)&=&-8b_2(x)=\frac{-1}{192}\sec^6(\frac{x}{2})+\frac{1}{64}\sec^8(\frac{x}{2})\\
G(2,3)&=&-8b_3(x)=\frac{1}{5760}\sec^6(\frac{x}{2})-\frac{1}{384}\sec^8(\frac{x}{2})+\frac{1}{256}\sec^{10}(\frac{x}{2})
\end{eqnarray*}
This gives the next three predictions.

\section{The Gromov-Witten Invariants of the Local $\cB\bZ_m$ Gerbe}
Let $\cX$ be the global quotient of the resolved conifold Tot$(\cO(-1)\oplus\cO(-1)\to \bP^1)$ by $\bZ_m$ acting fiberwise by $\xi_m$ and $\xi_m^{-1}$ respectively, where $\xi_m=e^{\frac{2\pi \sqrt{-1}}{m}}$. Then $\cX$ can be identified with Tot$(L_0\otimes\cO_{\cY_0}(-1)\oplus L_0^{-1}\otimes\cO_{\cY_0}(-1)\to \cY_0)$ where $L_0$ is the tautological bundle on $\cY_0$.

Define $C^\bullet_{\chi,d,\gamma}$ to be
$$C^\bullet_{\chi,d,\gamma}=\int_{[\Mbar^\bullet_{\chi,\gamma}(\cY_0,d)]^\vir}
e(R^1\pi_*F^*(L_0\otimes\cO_{\cY_0}(-1)\oplus L_0^{-1}\otimes\cO_{\cY_0}(-1)))$$
where
$$\pi:\cU\to \Mbar^\bullet_{\chi, \gamma}(\cY_0, d)$$
is the universal domain curve and
$$F:\cU\to \cY_0$$
is the evaluation map.
Let
\begin{eqnarray*}
J_{\chi,d,\gamma}=\{(\chi^1,\chi^2,\mubar,\gamma^1,\gamma^2)|\chi^1,\chi^2\in 2\bZ,|\mubar|=d,\gamma^1\sqcup  \gamma^2=\gamma,\\
-\chi^1+2l(\mubar)-\chi^2=-\chi,-\chi^1+2l(\mubar)\geq 0,-\chi^2+2l(\mubar)\geq 0\}.
\end{eqnarray*}
Then by the degeneration formula (see \cite{Abr-Fan})
$$C^\bullet_{\chi,d,\gamma}=\sum_{(\chi^1,\chi^2,\mubar,\gamma^1,\gamma^2)\in J_{\chi,d,\gamma}}K^{\bullet 0}_{\chi^1,\mubar,\gamma^1}Z_{\mubar}K^{\bullet 0}_{\chi^2,-\mubar,\gamma^2}$$
where $K^{\bullet 0}_{\chi,\mubar,\gamma }=\frac{1}{\am}\int_{[\Mbar^\bullet_{\chi, \gamma}(\cY_0, \mubar)]^{\vir}}e(V^0)$, $V^0$ is the obstruction bundle on $\Mbar^\bullet_{\chi, \gamma}(\cY_0, \mubar)$ and $Z_{\mubar}=|\Aut(\mubar)|m^{l(\mubar)}\prod_{i=1}^{l(\mubar)}\mu_i$.
Let
$$C^\bullet_{d}(\lambda;x)=\sum_{\chi,\gamma}\lambda^{-\chi+l(\gamma)}C^\bullet_{\chi,d,\gamma}\frac{x_\gamma}{\gamma!}$$
Then we have
$$C^\bullet_{d}(\lambda;x)=\sum_{|\mubar|=d}(-1)^{d-l'(\mubar)}K^{\bullet 0}_{\mubar}(\lambda;x)
Z_{\mubar}K^{\bullet 0}_{-\mubar}(\lambda;x)$$
where $K^{\bullet 0}_{\mubar}(\lambda;x)$ is defined in section 4.3. Recall that
$$K^{\bullet 0}_{\mubar}(\lambda;x)=G^\bullet_{\mubar}(\lambda;0;x)_m$$
Therefore
$$C^\bullet_{d}(\lambda;x)=\sum_{|\mubar|=d}(-1)^{d-l'(\mubar)}G^\bullet_{\mubar}(\lambda;0;x)_m
Z_{\mubar}G^\bullet_{-\mubar}(\lambda;0;x)_m$$
This finishes the calculation of the Gromov-Witten invariants of the local $\cB\bZ_m$ gerbe and hence finishes the proof of Theorem 3.

\begin{appendix}
\section{The Invertibility of $\tilde{\Phi}_d(\lambda;x)$}
We first recall the definition of $\tilde{\Phi}_d(\lambda;x)$:
$$\tilde{\Phi}_d(\lambda;x)=(\tilde{\Phi}_d^{(\mubar,s),\etabar}(\lambda;x))_{(\mubar,s)\in C_d,\etabar\in B_d}$$
where
$$\tilde{\Phi}_d^{(\mubar,s),\etabar}(\lambda;x)=\left\{\begin{array}{ll}0,&\textrm{if}|\etabar|>|\mubar|\\
Z_{w^s(-\etabar)}
\tilde{\Phi}^\bullet_{-w^s(-\etabar),\mubar}(-\frac{\sqrt{-1}s}{m}\lambda;\tilde{x})_m,&\textrm{if}|\etabar|=|\mubar|\\
\sum_{|\xibar|=|\mubar|-|\etabar|,l(\xibar)=l'(\xibar)}G^\bullet_{\xibar}(\lambda;0;x)_mZ_{w^s(-(\xibar\sqcup\etabar))}
\tilde{\Phi}^\bullet_{-w^s(-(\xibar\sqcup\etabar)),\mubar}(-\frac{\sqrt{-1}s}{m}\lambda;\tilde{x})_m
&\textrm{if}|\etabar|<|\mubar|\end{array} \right.$$
and
\begin{eqnarray*}
B_d&=&\{\etabar||\etabar|\leq d,l(\etabar)=l''(\etabar)\}\\
C_d&=&\{(\mubar,s)|\mu=\eta,s=s(\etabar),k_1=0,-w^s(\mubar\setminus \{(\mu_1,k_1)\})=\etabar\setminus \{(\eta_1,h_1)\},|\etabar|\leq d,l(\etabar)=l''(\etabar)\}.
\end{eqnarray*}
Therefore, in order to show that $\tilde{\Phi}_d(\lambda;x)$ is invertible, we only need to show the matrix
$$\tilde{\Phi}'_d(\lambda;x)=(\tilde{\Phi}_d^{(\mubar,s),\etabar}(\lambda;x))_{(\mubar,s)\in C'_d,\etabar\in B'_d}$$
is invertible for every $d\geq 1$, where
\begin{eqnarray*}
B'_d&=&\{\etabar||\etabar|= d,l(\etabar)=l''(\etabar)\}\\
C'_d&=&\{(\mubar,s)|\mu=\eta,s=s(\etabar),k_1=0,-w^s(\mubar\setminus \{(\mu_1,k_1)\})=\etabar\setminus \{(\eta_1,h_1)\},|\etabar|= d,l(\etabar)=l''(\etabar)\}.
\end{eqnarray*}
$$\tilde{\Phi}_d=\left(
\begin{array}{cccc}
\tilde{\Phi}'_d&*&\cdots&*\\
0&\tilde{\Phi}'_{d-1}&\ddots&\vdots\\
\vdots&\ddots&\ddots&*\\
0&\cdots&0&\tilde{\Phi}'_1
\end{array} \right).$$

Now we will show that the determinant of the matrix $ \tilde{\Phi}'_d(\lambda;\frac{x}{\lambda})|_{\lambda=x_2=\cdots=x_{m-1}=0}$ is nonzero. For convenience, we will denote
$ \tilde{\Phi}'_d(\lambda;\frac{x}{\lambda})|_{\lambda=x_2=\cdots=x_{m-1}=0}$ by $\tilde{\Phi}'_d(x_1)$. Recall that
$$\Phi^\bullet_{\mubar,\nubar}(\lambda;x)_m=\sum_{\chi\in 2\bZ,\chi\leq\min\{2l(\mubar),2l(\nubar)\},\gamma}\frac{\lambda^{-\chi+l(\mubar)+l(\nubar)+l(\gamma)}}
{(-\chi+l(\mubar)+l(\nubar))!}H^\bullet_{\chi,\gamma}(\mubar,\nubar)_m\frac{x_\gamma}{\gamma!}$$
So if the underlying partitions $\mu\neq \eta$, then the entry $\tilde{\Phi}'^{(\mubar,s),\etabar}_d(x_1)=0$. Therefore, in order to show that the determinant of the matrix $\tilde{\Phi}'_d(x_1)$ is nonzero, we only need to show that for any fixed partition $\mu^0$ of $d$, the determinant of the sub-matrix
$$\tilde{\Phi}'_{\mu^0}(x_1)=(\tilde{\Phi}_d^{(\mubar,s),\etabar}(x_1))_{(\mubar,s)\in C'_d,\etabar\in B'_d,\mu=\eta=\mu^0}$$
is nonzero.
$$\tilde{\Phi}'_d(x_1)=\left(
\begin{array}{cccc}
\tilde{\Phi}'_{\mu^0}(x_1)&0&\cdots&0\\
0&\tilde{\Phi}'_{\mu^1}(x_1)&\ddots&\vdots\\
\vdots&\ddots&\ddots&0\\
0&\cdots&0&\ddots
\end{array} \right).$$

Let $\hat{\etabar}=\etabar\setminus \{(\eta_1,h_1)\}$ and $\mu^0=\{\mu^0_1,\cdots,\mu^0_{l(\mu_0)}\}$. Let $c_{\mu^0}=\Gcd(m,\mu^0_1)$ and for any $1\leq i\leq \frac{m}{c_{\mu^0}}$ define $D^i_{\mu^0}$ to be
$$D^i_{\mu^0}=\{j\in\bZ|(i-1)c_{\mu^0}\leq j<ic_{\mu^0}\}.$$
For fixed $\hat{\etabar}^0$ and $1\leq i\leq \frac{m}{c_{\mu^0}}$ define
\begin{eqnarray*}
B^i_{\mu^0,\hat{\etabar}^0}&=&\{\etabar|\eta=\mu^0,l(\etabar)=l''(\etabar),\hat{\etabar}=\hat{\etabar}^0,h_1\in D^i\}\\
C^i_{\mu^0,\hat{\etabar}^0}&=&\{(\mubar,s)|\mu=\eta,s=s(\etabar),k_1=0,-w^s(\mubar\setminus \{(\mu_1,k_1)\})=\hat{\etabar},\eta=\mu^0,l(\etabar)=l''(\etabar),\hat{\etabar}=\hat{\etabar}^0,h_1\in D^i\}.
\end{eqnarray*}
Then we define the sub-matrix $\tilde{\Phi}'_{\mu^0,\hat{\etabar}^0,i}(x_1)$ of $\tilde{\Phi}'_{\mu^0}(x_1)$ to be
$$\tilde{\Phi}'_{\mu^0,\hat{\etabar}^0,i}(x_1)=(\tilde{\Phi}_d^{(\mubar,s),\etabar}(x_1))_{(\mubar,s)\in C^i_{\mu^0,\hat{\etabar}^0},\etabar\in B^i_{\mu^0,\hat{\etabar}^0}}$$
\begin{lemma}
The determinant of the matrix $\tilde{\Phi}'_{\mu^0,\hat{\etabar}^0,i}(x_1)$ is nonzero
\end{lemma}
\begin{proof}
If we view the entries of the matrix $\tilde{\Phi}'_{\mu^0,\hat{\etabar}^0,i}(x_1)$ as power series in $\bC[[x_1]]$, then the lowest degree term of $\tilde{\Phi}_d^{(\mubar,s),\etabar}(x_1)$ is
$$Z_{w^s(-\etabar)}(-\frac{\sqrt{-1}s}{m})^{n(h_1)}H^\bullet_{2l(\mubar),\gamma(h_1)}(\mubar,-w^s(-\etabar))_m
\frac{(\sqrt{-1}^{1-\frac{2}{m}}x_1)^{n(h_1)}}{n(h_1)!}$$
where $n(h_1)=w^s_{\eta_1}(-h_1)=\bar{h}_1\in \{0,\cdots,c_{\mu^0}-1\}$ is \emph{independent} of $s=s(\etabar)$ and $\gamma(h_1)=(x_1,\cdots,x_1)$ with $l(\gamma)=n(h_1)$. Also note that $|\Aut(\etabar)|=|\Aut(w^s(-\etabar))|=|\Aut(-w^s(-\etabar))|$. Therefore, by the calculation in section 3.2, we have
\begin{eqnarray*}
Z_{w^s(-\etabar)}H^\bullet_{2l(\mubar),\gamma(h_1)}(\mubar,-w^s(-\etabar))_m&=&Z_{\etabar}\frac{1}{|\Aut(\etabar)|
|\Aut(\mubar)|}|\Aut(\etabar)|m^{-l(\etabar)}\eta_1^{n(h_1)}\prod_{i=1}^{l(\etabar)}\eta_i^{-1}\\
&=&\frac{|\Aut(\etabar)|}{|\Aut(\mubar)|}\eta_1^{n(h_1)}
\end{eqnarray*}
Therefore,
\begin{eqnarray*}
&&Z_{w^s(-\etabar)}(-\frac{\sqrt{-1}s}{m})^{n(h_1)}H^\bullet_{2l(\mubar),\gamma(h_1)}(\mubar,-w^s(-\etabar))_m
\frac{(\sqrt{-1}^{1-\frac{2}{m}}x_1)^{n(h_1)}}{n(h_1)!}\\
&=&(-\frac{\sqrt{-1}s}{m})^{n(h_1)}\frac{|\Aut(\etabar)|}{|\Aut(\mubar)|}
\frac{(\sqrt{-1}^{1-\frac{2}{m}}\eta_1x_1)^{n(h_1)}}{n(h_1)!}
\end{eqnarray*}
So the lowest degree term of $\det(\tilde{\Phi}'_{\mu^0,\hat{\etabar}^0,i}(x_1))$ is
\begin{eqnarray*}
&&\left(\prod_{(\mubar,s)\in C^i_{\mu^0,\hat{\etabar}^0}}\frac{1}{|\Aut(\mubar)|}\right)\left(\prod_{\etabar\in B^i_{\mu^0,\hat{\etabar}^0}}\frac{|\Aut(\etabar)|(\sqrt{-1}^{1-\frac{2}{m}}\eta_1x_1)^{n(h_1)}}{n(h_1)!}\right)\\
&&\cdot\det\left((-\frac{\sqrt{-1}s}{m})^{n(h_1)}\right)_{(\mubar,s)\in C^i_{\mu^0,\hat{\etabar}^0},\etabar\in B^i_{\mu^0,\hat{\etabar}^0}}
\end{eqnarray*}
Thus we only need to show $\det\left((-\frac{\sqrt{-1}s}{m})^{n(h_1)}\right)_{(\mubar,s)\in C^i_{\mu^0,\hat{\etabar}^0},\etabar\in B^i_{\mu^0,\hat{\etabar}^0}}$ is nonzero. But this is a Vandermonde matrix with different $s$ in different rows. So its determinant is nonzero.
\end{proof}

Now for any fixed column $\alpha_{\etabar^0}$ of $\tilde{\Phi}'_{\mu^0}(x_1)$, there is a unique sub-matrix $\tilde{\Phi}'_{\mu^0,\hat{\etabar}^0,i}(x_1)$ that intersects with this column. Then the degrees of the entries that lie in the intersection of $\alpha_{\etabar^0}$ and $\tilde{\Phi}'_{\mu^0,\hat{\etabar}^0,i}(x_1)$ are $n(h^0_1)=\bar{h}^0_1\in \{0,\cdots,c_{\mu^0}-1\}$. By the convention of the order of a $\bZ_m$-weighted partition (see section 2.1), the degrees of the other entries of $\alpha_{\etabar^0}$ are greater or equal to $n(h^0_1)$ (note that $k_1$ is always 0). And the equality holds for an entry $\tilde{\Phi}_d^{(\mubar',s'),\etabar^0}(x_1)$ among those entries if and only if the following conditions hold:
\begin{enumerate}
\item There exists a $j>1$ such that $\eta^0_j=\eta^0_1,\bar{h}^0_1=\bar{h}^0_j$ and $h^0_j>h^0_1$, where $\bar{h}^0_1$ and $\bar{h}^0_j$ denote $h^0_1$(mod $c_{\mu^0}$) and $h^0_j$(mod $c_{\mu^0}$) respectively.

\item Let $\hat{\etabar}'^{0}=\etabar^0\setminus \{(\eta^0_j,h^0_j)\}$, then $-w^{s'}(\mubar'\setminus \{(\mu'_1,k'_1)\})=\hat{\etabar}'^{0}$ and $-h^0_j+\eta^0_js'=-\bar{h}^0_j\in \bZ_m$.
\end{enumerate}
Then there is a unique sub-matrix $\tilde{\Phi}'_{\mu^0,\hat{\etabar}'^{0},i'}(x_1)$ that intersects with the row that contains $\tilde{\Phi}_d^{(\mubar',s'),\etabar^0}(x_1)$. By condition (2), $h^0_j\in D^{i'}$. It is easy to see that every entry that lies in the intersection of $\alpha_{\etabar^0}$ and a row of $\tilde{\Phi}'_{\mu^0,\hat{\etabar}'^{0},i'}(x_1)$ has degree $n(h^0_1)$.
$$\left(
\begin{array}{ccccccc}
&&&*&\cdots&\cdots&\cdots\\
&\tilde{\Phi}'_{\mu^0,\hat{\etabar}^0,i}(x_1)&&*&\cdots&\cdots&\cdots\\
&&&*&\cdots&\cdots&\cdots\\
*&*&*&\ddots&*&*&*\\
*&*&x_1^{n(h^0_1)}&*&&&\\
\vdots&\vdots&\vdots&*&&\tilde{\Phi}'_{\mu^0,\hat{\etabar}'^{0},i'}(x_1)&\\
*&*&x_1^{n(h^0_1)}&*&&&
\end{array}\right)$$
For every column $\alpha_{\etabar'}$ that intersects with $\tilde{\Phi}'_{\mu^0,\hat{\etabar}'^{0},i'}(x_1)$, $\bar{h}'_1\leq \bar{h}^0_1$ by our convention. But $\bar{h}'_1$ can not be equal to $\bar{h}^0_1$, because otherwise $h'_1=h^0_j$ by condition (1) and the fact that $h^0_j,h'_1\in D^{i'}$. Hence $\etabar^0=\etabar'^{0}$, a contradiction. So we have $\bar{h}'_1< \bar{h}^0_1$. In other words, condition (1) can not be satisfied for any entry that lies in $\alpha_{\etabar'}$ but does not lie in $\tilde{\Phi}'_{\mu^0,\hat{\etabar}'^{0},i'}(x_1)$. So the degrees of these entries are strictly greater than $n(h'_1)$. By Lemma A.1, we can use the elementary transforms for matrices to convert $\tilde{\Phi}'_{\mu^0,\hat{\etabar}'^{0},i'}(x_1)$ to $\Psi_{\mu^0,\hat{\etabar}'^{0},i'}(x_1)$ such that if we pick the degree $n(h_1')$ terms of the entries in the column $\alpha_{\etabar'}\cap \Psi_{\mu^0,\hat{\etabar}'^{0},i'}(x_1)$ of $\Psi_{\mu^0,\hat{\etabar}'^{0},i'}(x_1)$ for every $\alpha_{\etabar'}\cap \Psi_{\mu^0,\hat{\etabar}'^{0},i'}(x_1)\neq \emptyset$ to form a matrix, then this matrix is of the following form
$$\left(
\begin{array}{ccccc}
1&0&\cdots&\cdots&0\\
0&x_1&0&\cdots&0\\
\vdots&0&\ddots&\ddots&\vdots\\
\vdots&\vdots&\ddots&\ddots&0\\
0&0&\cdots&0&x_1^{|B^i_{\mu^0,\hat{\etabar}'^0}|-1}
\end{array}\right)$$
In this process, the following two properties do not change: (a) For every column $\alpha_{\etabar'}$ that intersects with $\tilde{\Phi}'_{\mu^0,\hat{\etabar}'^{0},i'}(x_1)$, the degrees of the entries that lie in $\alpha_{\etabar'}$ but do not lie in $\tilde{\Phi}'_{\mu^0,\hat{\etabar}'^{0},i'}(x_1)$ are strictly greater than $n(h'_1)$; (b) For every column $\alpha_{\etabar}$, let $\tilde{\Phi}'_{\mu^0,\hat{\etabar},i}(x_1)$ be the sub-matrix that intersects with $\alpha_{\etabar}$, then the degrees of the entries that lie in $\alpha_{\etabar}$ but do not lie in $\tilde{\Phi}'_{\mu^0,\hat{\etabar},i}(x_1)$ are no less than (or strictly greater than) $n(h_1)$. Now we can use the columns that intersect with $\Psi_{\mu^0,\hat{\etabar}'^{0},i'}(x_1)$ to cancel the degree $n(h^0_1)$ terms of the entries that lie in the intersection of $\alpha_{\etabar^0}$ and rows of $\Psi_{\mu^0,\hat{\etabar}'^{0},i'}(x_1)$.
$$\left(
\begin{array}{ccccccc}
&&&*&\cdots&\cdots&\cdots\\
&\tilde{\Phi}'_{\mu^0,\hat{\etabar}^0,i}(x_1)&&*&\cdots&\cdots&\cdots\\
&&&*&\cdots&\cdots&\cdots\\
*&*&*&\ddots&*&*&*\\
*&*&x_1^{n(h^0_1)}&*&&&\\
\vdots&\vdots&\vdots&*&&\Psi_{\mu^0,\hat{\etabar}'^{0},i'}(x_1)&\\
*&*&x_1^{n(h^0_1)}&*&&&
\end{array}\right)$$
Since property (a) is preserved, this process will not change the degree $n(h^0_1)$ terms of the entries that lie in $\alpha_{\etabar^0}$ but do not lie in the rows of $\Psi_{\mu^0,\hat{\etabar}'^{0},i'}(x_1)$. In particular, the lowest degree term of $\det\left(\tilde{\Phi}'_{\mu^0,\hat{\etabar}^0,i}(x_1)\right)$ does not change. Since property (b) is preserved, we can repeat this process until for every $\etabar$ all the degree $n(h_1)$ entries in $\alpha_{\etabar}$ lie in the unique sub-matrix $\tilde{\Phi}'_{\mu^0,\hat{\etabar},i}(x_1)$ (or $\Psi_{\mu^0,\hat{\etabar},i}(x_1)$ if $\tilde{\Phi}'_{\mu^0,\hat{\etabar},i}(x_1)$ has been changed) that intersects with $\alpha_{\etabar}$. Therefore, if we denote the result matrix of $\tilde{\Phi}'_{\mu^0}(x_1)$ of this process by $\Psi_{\mu^0}(x_1)$, then the lowest degree term of $\det\left(\Psi_{\mu^0}(x_1)\right)$ is the product of that of $\tilde{\Phi}'_{\mu^0,\hat{\etabar},i}(x_1)$ or $\Psi_{\mu^0,\hat{\etabar},i}(x_1)$ which is nonzero by Lemma A.1. So $\det\left(\tilde{\Phi}'_{\mu^0}(x_1)\right)$ is nonzero.

In conclusion, the matrix $\tilde{\Phi}_d(\lambda;x)$ is invertible.

\end{appendix}

\end{document}